\documentclass[a4paper,reqno]{amsart}
\usepackage{amsaddr}
\pdfoutput=1
\usepackage[utf8]{inputenc}

\usepackage{pgf}
\usepackage{tikz}
\usepackage{pgfplots}
\pgfplotsset{compat=newest}
\newlength\fwidth
\newlength\fheight

\usepackage[margin=10pt,font=small,labelfont=bf,
labelsep=endash]{caption}
\usepackage{cite}

\usepackage{enumitem}

\usepackage{amsmath}
\usepackage{amssymb}
\usepackage{amsthm}
\usepackage{mathrsfs}

\usepackage[hidelinks]{hyperref}
\usepackage{cleveref}

\newtheorem{thm}{Theorem}

\newtheorem{lemma}[thm]{Lemma}

\newtheorem{definition}[thm]{Definition}

\theoremstyle{remark}
\newtheorem{remark}[thm]{Remark}
\newtheorem*{remark*}{Remark}

\newcommand{\dd}{\mathrm{d}}
\newcommand{\R}{\ensuremath{\mathbb{R}}} 
\newcommand{\C}{\ensuremath{\mathbb{C}}} 
\newcommand{\N}{\ensuremath{\mathbb{N}}} 
\newcommand{\Z}{\ensuremath{\mathbb{Z}}} 
\newcommand{\sspan}[1]{\ensuremath{\langle #1\rangle}} 

\newcommand{\T}{\ensuremath{\mathbb{T}}}
\newcommand{\lap}{\ensuremath{\mathcal{L}}}

\newcommand{\principalV}{\ensuremath{\mathsf{PV}}}
\newcommand{\conv}{*} 

\DeclareMathOperator{\divergence}{div}

\DeclareMathOperator*{\esssup}{ess\,sup}

\newcommand{\conj}[1]{\overline{#1}}

\begin{document}
\title{Stability and bifurcation for the Kuramoto model}
\author{Helge Dietert}
\address{Department of Pure Mathematics and Mathematical Statistics\\
University of Cambridge\\
Wilberforce Road\\
Cambridge CB3 0WA, UK}
\email{H.G.W.Dietert@maths.cam.ac.uk}
\thanks{The author was supported by the UK Engineering and Physical
Sciences Research Council (EPSRC) grant EP/H023348/1 for the
University of Cambridge Centre for Doctoral Training, the Cambridge
Centre for Analysis.}

\begin{abstract}
  We study the mean-field limit of the Kuramoto model of globally
  coupled oscillators. By studying the evolution in Fourier space and
  understanding the domain of dependence, we show a global stability
  result. Moreover, we can identify function norms to show damping of
  the order parameter for velocity distributions and perturbations in
  $\mathcal{W}^{n,1}$ for $n > 1$. Finally, for sufficiently regular
  velocity distributions we can identify exponential decay in the
  stable case and otherwise identify finitely many eigenmodes. For
  these eigenmodes we can show a center-unstable manifold reduction,
  which gives a rigorous tool to obtain the bifurcation behaviour. The
  damping is similar to Landau damping for the Vlasov equation.
\end{abstract}
\maketitle

{\small

  \vspace{1em}\noindent
}

\tableofcontents

\section{Introduction}

\subsection{Overview}
The Kuramoto model for globally coupled oscillators is a prototype
for synchronisation behaviour. Introduced 40 years ago to model
chemical instabilities \cite{kuramoto1975self,kuramoto1984chemical},
it has found many other applications, see
\cite{strogatz2000from,acebron2005synchronization}.

The finite model consists of $N$ oscillators with natural frequencies
$(\omega_i)_{i=1}^{N}$ whose phase $(\theta_i)_{i=1}^{N}$ evolves as
\begin{align*}
  \partial_t \theta_i(t)
  &= \omega_i
  + \frac{K}{N} \sum_{j=1}^{N} \sin(\theta_j(t)-\theta_i(t)) \\
  &= \omega_i + \frac{K}{2i}
  \left(\eta(t)\, e^{-i\theta_i(t)} - \conj{\eta(t)}\, e^{i\theta_i(t)}\right),
\end{align*}
for $i=1,\dots ,N$, where $K$ is the coupling constant and
\begin{equation*}
  \eta(t) = \frac{1}{N} \sum_{j=1}^{N} e^{i\theta_j(t)}
\end{equation*}
is the order parameter and the main observable of the system.

This has the structure of a kinetic particle system where the phase
angle $\theta$ corresponds to the spatial position and the natural
frequency $\omega$ to the velocity. Following this structure, we will
use the terminology of spatial and velocity distribution.

We study the mean-field limit for $N \to \infty$, cf.
\cite{strogatz1991stability,lancellotti2005vlasov}. Main features of
the evolution were already identified in early studies
\cite{kuramoto1984chemical,kuramoto1975self} and
\cite{strogatz2000from} contains a more precise review. The spatially
homogeneous distribution is a trivial stationary solution with
$\eta = 0$ and is also called the (fully) incoherent state. The system
then has a critical coupling constant $K_c$ so that for $K < K_c$ the
incoherent state appears to be stable and for larger $K$ the system
bifurcates to partially locked states.

Even though the order parameter suggests a stable behaviour for
$K < K_c$, the linearised evolution operator in normal physical
function spaces has a continuous spectrum along the imaginary axis and
thus does not show decay directly
\cite{strogatz1991stability}. Already in 1992 it was realised that the
behaviour is similar to Landau damping in plasma physics
\cite{strogatz1992coupled}.

Here our starting point is to study the model in Fourier variables for
the phase $\theta$ and the frequency $\omega$. By understanding the
domain of dependence, we can prove a global stability result: There
exists $K_{ec}$ such that for couplings $K < K_{ec}$ the evolution is
controlled by
\begin{equation*}
  \int_0^{\infty} |\eta(t)|^2 \dd t < \infty.
\end{equation*}
For general velocity distributions $K_{ec} < K_c$, but for the
Lorentzian and Gaussian distribution $K_{ec} = K_{c}$. As far as the
author is aware, this is the first global result controlling the order
parameter.

For the local behaviour of the spatially homogeneous state, we first
note that the theory of the Volterra equation
\cite{gripenberg1990volterra} gives a rigorous stability result for
the linearised evolution of the order parameter $\eta$. In particular,
it shows exponential decay when the Fourier transform decays
exponentially and algebraic decay $t^{-n}$ when the Fourier transform
decays like $t^{-n}$. In particular the condition for algebraic decay
is satisfied if the velocity distribution and the perturbation are in
the Sobolev space $\mathcal{W}^{n,1}$.

By working in Fourier variables, we can now define norms for the
perturbation in which the linearised evolution has the same decays as
the order parameter. Moreover, for exponential decay or algebraic
decay $t^{-n}$ for $n > 1$, we are able to extend the linear stability
to nonlinear stability of perturbations. In particular, this shows
that the order parameter decays as predicted by the linear analysis.

In case the linear stability criterion fails and we have exponential
regularity, the linear analysis allows us to identify finitely many
eigenmodes. By constructing suitable norms we are now able to prove a
center-unstable manifold reduction of the dynamics. In particular,
this allows the rigorous derivation of the bifurcation and shows
stability of the partially locked states in the supercritical
bifurcation.

We stress that our results and used norms are independent of the
particular velocity distribution. Thus the stability result and the
center-unstable manifold reduction hold for all sufficiently regular
distributions. To the best of the author's knowledge this result was
unknown.

For Gaussian and rational velocity distributions a similar stability
and bifurcation result has been obtained by Chiba
\cite{chiba2013ergodic,chiba2011center}. The dynamics of the system
can also be reduced by using a symmetry reduction, the so called
Ott-Antonsen ansatz \cite{ott2011comment,ott2009long,ott2008low}. For
special rational velocity distribution ($(1+\omega^2)^{-1}$,
$(1+\omega^4)^{-1}$) \cite{ott2008low} and for the Bi-Cauchy
distribution \cite{martens2009exact} this has been used to obtain the
dynamics. In case of identical oscillators (i.e. the frequency
distribution is concentrated at one frequency), it has been proved
that the system will asymptotically be fully synchronised
\cite{ha2010complete}. This result has been extended to compactly
supported velocity distributions with sufficiently coherent initial
data and large enough coupling constant
\cite{carrillo2014contractivity}.

The nonlinear stability and bifurcation have been studied in the
modified system with white noise \cite{sakaguchi1988cooperative}. In
this case the mean-field limit becomes a Fokker-Planck equation and
the results have been obtained in the 1990s
\cite{strogatz1991stability,bonilla1992nonlinear,crawford1994amplitude}. The
concrete bifurcation behaviour for different velocity distributions
and different interactions has been studied in more detail
\cite{crawford1999synchronization,crawford1995scaling,acebron1998breaking,bonilla1998time},
see \cite{acebron2005synchronization} for a review.

The above results show that the Kuramoto model is an interesting model
to understand Landau damping where we can obtain more results than for
the classical Vlasov equation for plasmas
\cite{bedrossian2013landau,mouhot2011landau}. Other models are the 2d
Euler equation \cite{bedrossian2013inviscid} and the Vlasov-HMF model,
where recently stability has been proved with Sobolev norms
\cite{faou2014landau}.

\begin{remark*}
  After finishing the work, the author got aware of the independent
  work by Fernandez, Gérard-Varet, and Giacomin
  \cite{fernandez2014landau}. This very interesting work shows the
  nonlinear stability for perturbations in the class $C^n$ ($n \ge 4$)
  and follows more closely \cite{faou2014landau}. After finishing a
  first draft another preprint \cite{benedetto2014exponential}
  appeared showing for analytic initial data exponential convergence
  to the incoherent state if the coupling constant $K$ is small
  enough.
\end{remark*}

\subsection{Strategy and detailed results}

For finitely many oscillators with phase angles $(\theta_i)_{i=1}^{N}$
and frequencies $(\omega_i)_{i=1}^{N}$ we can look at the empirical
measure
$\rho_{N,t}(\theta,\omega) = N^{-1} \sum_{i=1}^{N}
\delta_{\theta=\theta_i(t)} \delta_{\omega=\omega_i(t)}$
in the phase space $\Gamma = \T \times \R$, where $\T$ is the torus
$[0,2\pi)$ for the phase angles. The set of probability measures on a
space $X$ is denoted by $\mathcal{M}(X)$ and in the mean-field limit
$N \to \infty$ we suppose that the initial distribution converges
weakly to a probability distribution
$\rho_{in} \in \mathcal{M}(\Gamma)$. Then for all later times
$t \in \R^+$ the distribution converges weakly to a probability
distribution $\rho_t$ and satisfies a PDE in a weak sense. For the
limiting case $N \to \infty$ we can thus study the family $\rho_t$
whose evolution is uniquely determined by the PDE. Given initial data
$\rho_{in}$, we call this family $\rho_{t}$ the solution to the
Kuramoto equation or the mean-field solution.

Assuming a density function $f(t,\cdot,\cdot)$ of $\rho_t$, the system
evolves formally as
\begin{equation*}
  \begin{cases}
    \partial_t f(t,\theta,\omega) + \partial_\theta
    \left[ \left(\omega + \frac{K}{2i}
        (\eta(t)\, e^{-i\theta} - \conj{\eta(t)}\, e^{i\theta})
        \right) f(t,\theta,\omega) \right] = 0, \\
    \eta(t) = \int_{\theta=0}^{2\pi} e^{i\theta} \int_{\R}
    f(t,\theta,\omega) \dd \omega \dd \theta.
  \end{cases}
\end{equation*}
In the general case, we rigorously understand this PDE in a weak sense
for measures. As a Cauchy problem, we will always assume that the
initial data $\rho_{in}$ are given at time $t=0$ and then the
solutions $\rho_{\cdot}$ will be defined for all $t \in \R^+$.

The velocity distribution $g$ is a probability distribution
$\mathcal{M}(\R)$ which is constant in time, so that
$g(A) = \rho_t(\T \times A)$ holds for all $A \in \mathcal{B}(\R)$ and
time $t$, where $\mathcal{B}(\R)$ denotes the Borel sets of $\R$.

Working formally with a density $f$, the velocity distribution has a
density $g$ and we can take the Fourier series for the spatial modes
\begin{equation*}
  \hat{f}(t,l,\omega) = \int_{\theta\in\T} e^{il\theta} f(t,\theta,\omega)
  \dd \theta
\end{equation*}
for $l\in\Z$. These then evolve as
\begin{equation*}
  \begin{cases}
    \partial_t \hat{f}(t,l,\omega) = i l \omega\, \hat{f}(t,l,\omega)
    + \frac{Kl}{2} \left[\eta(t)\, \hat{f}(t,l-1,\omega)
      - \conj{\eta(t)}\, \hat{f}(t,l+1,\omega) \right], \\
    \eta(t) = \int_{\R} \hat{f}(t,1,\omega) \dd \omega.
  \end{cases}
\end{equation*}

Since $f$ is non-negative and its marginal $g(\omega)
=\int_{\theta\in\T} f(t,\theta,\omega) \dd \theta$ is constant in
time, we find for all time $t$
\begin{equation*}
  \hat{f}(t,0,\omega) = g(\omega)
  \text{ and } |\hat{f}(t,l,\omega)| \le g(\omega)
  \text{ for all $l\in\Z$}.
\end{equation*}

Moreover, since $f$ is real,
$\hat{f}(l,\omega) = \conj{\hat{f}(-l,\omega)}$ so that it suffices to
consider $l\in\N$.

The spatially homogeneous distribution is characterised by vanishing
non-zero moments, so that the stability of the spatially homogeneous
distribution is equivalent to the decay of $\hat{f}$ restricted to
$l \ge 1$. The evolution of the restriction satisfies
\begin{equation*}
  \begin{cases}
    \partial_t \hat{f}(t,1,\omega) = i \omega\, \hat{f}(t,1,\omega) + \frac{K}{2}
    \left[\eta(t)\, g(\omega) - \conj{\eta(t)}\, \hat{f}(t,2,\omega) \right] \\
    \partial_t \hat{f}(t,l,\omega) = i l \omega\, \hat{f}(t,l,\omega) + \frac{Kl}{2}
    \left[\eta(t)\, \hat{f}(t,l-1,\omega) - \conj{\eta(t)}\, \hat{f}(t,l+1,\omega)
    \right] \text{ for $l\ge 2$} \\
    \eta(t) = \int_{\R} \hat{f}(t,1,\omega) \dd \omega.
  \end{cases}
\end{equation*}

Exploiting that the non-linear interaction is skew-Hermitian, we find
\begin{equation*}
  \partial_t \int_{\R} \sum_{l=1}^{\infty} \frac{1}{l}
  |\hat{f}(t,l,\omega)|^2 g^{-1}(\omega) \dd \omega = K |\eta(t)|^2.
\end{equation*}
With $K>0$ the system therefore is not dissipative and the convergence
to zero must happen in a weaker sense. Moreover, it shows
that the model is not time-reversible, as $K$ needs to change sign as
well. Finally, for negative $K$ this already shows global stability as
then $\int |\eta(t)|^2 \dd t$ is bounded. Hence the interesting case
is $K \ge 0$ which we assume in the following.

In order to understand the evolution of the spatial modes $l \ge 1$,
we also Fourier transform the velocity variable $\omega$ as
\begin{equation}
  \label{eq:fourier-u}
  u(t,l,\xi) = \int_{(\theta,\omega)\in\Gamma} e^{il\theta}
  e^{i\xi\omega}
  \rho_t(\dd \theta,\dd \omega)
\end{equation}
and
\begin{equation}
  \label{eq:fourier-g}
  \hat{g}(\xi) = \int_{\omega \in \R} e^{i\xi\omega} g(\dd \omega).
\end{equation}
The difference to the spatially homogeneous state is again captured by
the restriction to $l \ge 1$. Hence the evolution of
$u(t,\cdot,\cdot) \in C(\N \times \R)$ describes the evolution of the
perturbation and is determined by
\begin{equation}
  \label{eq:evolution-double-fourier}
  \begin{cases}
    \partial_t u(t,1,\xi) = \partial_\xi u(t,1,\xi) +
    \frac{K}{2} \left[\eta(t)\, \hat{g}(\xi) - \conj{\eta(t)}\, u(t,2,\xi) \right], \\
    \partial_t u(t,l,\xi) = l\partial_\xi u(t,l,\xi) +
    \frac{Kl}{2} \left[\eta(t)\, u(t,l-1,\xi) - \conj{\eta(t)}\, u(t,l+1,\xi)
    \right] \text{ for $l\ge2$}, \\
    \eta(t) = u(t,1,0).
  \end{cases}
\end{equation}

In this formulation the free transport always moves
$\xi \mapsto u(t,l,\xi)$ to the left and the interaction at
$(t,l,\xi)$ is given by $u(t,1,0) = \eta(t)$ and $u(t,l\pm 1,\xi)$.
Hence with $M \ge 0$, the region $\N \times \{ \xi : \xi \ge -M \}$ is
its own domain of dependence. This already suggests to use norms
focusing on $\xi \ge 0$ in order to show convergence.

For the global energy estimate, we note that the interaction is
skew-Hermitian except for the first mode. Thus we consider the energy
functional
\begin{equation}
  \label{eq:energy-functional}
  I(t) = \int_{\xi=0}^{\infty} \sum_{l\ge 1} \frac{1}{l}
  |u(t,l,\xi)|^2 \phi(\xi) \dd \xi,
\end{equation}
where $\phi$ is an increasing weight. By balancing the gain of the
first linear interaction term with the decay by the increasing weight,
we can show our first main result.
\begin{thm}
  \label{thm:energy-stability}
  For a velocity distribution $g \in \mathcal{M}(\R)$ suppose
  $\int_0^\infty |\hat{g}(\xi)| \dd \xi < \infty$ and let
  \begin{equation*}
    K_{ec} = \frac{2}{\int_{\xi=0}^{\infty} |\hat{g}(\xi)| \dd \xi}.
  \end{equation*}
  If the coupling constant $K$ satisfies $K < K_{ec}$, then there
  exists a finite constant $c > 0$ and a bounded increasing weight
  $\phi \in C^1(\R^+)$ with $\phi(0) = 1$ so that for a solution to
  the Kuramoto equation with velocity marginal $g$ the energy
  functional $I(t)$ defined by \Cref{eq:energy-functional} satisfies
  \begin{equation*}
    I(t) + c \int_0^t |\eta(s)|^2 \dd s
    \le I(0)
  \end{equation*}
  for all $t \in \R^+$.  In particular this shows that $I$ is
  non-increasing and the order parameter $\eta$ satisfies
  \begin{equation*}
    \int_0^\infty |\eta(s)|^2 \dd s \le
    c^{-1} I(0) < \infty.
  \end{equation*}
  If, moreover, the initial distribution $\rho_{in}$ has a density
  $f_{in} \in L^2(\Gamma)$, then
  $I(0) \le \| \phi \|_{\infty} \| f_{in} \|_2^2$.
\end{thm}

In order to quantify the decay and stability we introduce, for a
non-negative weight $\phi$, the weighted spaces $L^1(X,\phi)$ and
$L^\infty(X,\phi)$ for functions $f: X \mapsto \C$ with finite norm
\begin{equation*}
  \| f \|_{L^1(X,\phi)} = \int_{x \in X} |f(x)| \phi(x) \dd x
\end{equation*}
and
\begin{equation*}
  \| f \|_{L^\infty(X,\phi)} = \esssup_{x \in X} |f(x)| \phi(x),
\end{equation*}
respectively. As weights we will use $\exp_a$ defined as
$\exp_a(x) = e^{ax}$ and $p_{A,b}$ defined as $p_{A,b}(x) = (A+x)^b$
with the special case $p_b = p_{1,b}$.

In Fourier variables the linear evolution is determined by the PDE
\begin{equation}
  \label{eq:fourier-linear}
  \begin{cases}
    \partial_t u(t,1,\xi) = \partial_{\xi} u(t,1,\xi)
    + \frac{K}{2} u(t,1,0) \hat{g}(\xi), \\
    \partial_t u(t,l,\xi) = l\partial_\xi u(t,l,\xi) \text{ for $l\ge2$}.
  \end{cases}
\end{equation}
Here all modes decouple and only the first mode differs through the
interaction. We denote its evolution operator by $e^{tL}$,
i.e. $t,l,\xi \mapsto (e^{tL}u_{in})(l,\xi)$ is the weak solution to
\Cref{eq:fourier-linear} in $C(\R^+ \times \N \times \R)$ with initial
data $u_{in}$ given at time $t=0$. As in the case of the Vlasov
equation \cite{penrose1960electrostatic} the order parameter under the
linear evolution is given by a Volterra equation.
\begin{lemma}
  \label{thm:linear-evolution}
  Let $\hat{g} \in C(\R)$. Then the evolution operator $e^{tL}$ is
  well-defined from continuous initial data to the unique continuous
  solution. For $u_{in} \in C(\N \times \R)$ let
  $\nu(t) = (e^{tL}u_{in})(1,0)$. Then $\nu$ satisfies the Volterra
  equation
  \begin{equation}
    \label{eq:linear-order-volterra}
    \nu(t) + (k \conv \nu)(t) = u_{in}(1,t)
  \end{equation}
  with the convolution kernel
  \begin{equation}
    \label{eq:linear-order-kernel}
    k(t) = - \frac{K}{2} \hat{g}(t)
  \end{equation}
  and
  \begin{equation}
    \label{eq:linear-order-first-mode}
    (e^{tL} u_{in})(1,\xi) = u_{in}(1,\xi + t) + \int_0^t \frac{K}{2}
    \nu(s) \hat{g}(\xi+t-s) \dd s
  \end{equation}
  and for $l \ge 2$
  \begin{equation}
    \label{eq:linear-order-higher-modes}
    (e^{tL} u_{in})(l,\xi) = u_{in}(l,\xi + tl).
  \end{equation}
\end{lemma}
Here the convolution $f \conv g$ between two functions $f,g \in
L^1_{loc}(\R^+)$ is defined for $t \in \R^+$ by $(f \conv g)(t) =
\int_0^t f(t-s) g(s) \dd s$.

The solution to the Volterra equation is given by the resolvent kernel
$r \in L^1_{loc}(\R)$ as
\begin{equation*}
  \nu(t) = u_{in}(1,t) - (r \conv u_{in}(1,\cdot))(t).
\end{equation*}

Hence the order parameter $\nu$ under the linear evolution can decay
if $u_{in}(1,\xi)$ is decaying as $\xi$ increases. We can express this
for $a \ge 0$ as
\begin{equation*}
  \| \nu \|_{L^\infty(\R^+,\exp_a)}
  \le (1 + \| r \|_{L^1(\R^+,\exp_a)})
  \| u_{in}(1,\cdot) \|_{L^\infty(\R^+,\exp_a)}
\end{equation*}
and for $A \ge 1$ and $b \ge 0$ as
\begin{equation*}
  \| \nu \|_{L^\infty(\R^+,p_{A,b})}
  \le (1 + \| r \|_{L^1(\R^+,p_{b})})
  \| u_{in}(1,\cdot) \|_{L^\infty(\R^+,p_{A,b})}.
\end{equation*}

With the same decay in $\hat{g}$ we have a sharp criterion for the
resolvent kernel by the Paley-Wiener theorem for Volterra equations
\cite[Chapter 2, Thm 4.1]{gripenberg1990volterra} and Gel'fand's
theorem \cite[Chapter 4, Thm 4.3]{gripenberg1990volterra}.
\begin{thm}
  \label{thm:order-parameter-stability}
  If for $a \in \R$ we have $\hat{g} \in L^1(\R,\exp_a)$, then
  $r \in L^1(\R, \exp_a)$ is equivalent to $(\lap k)(z) \not = -1$ for
  all $z \in \C$ with $\Re z \ge -a$.

  If for $b \ge 0$ we have $\hat{g} \in L^1(\R,p_{b})$, then
  $r \in L^1(\R, p_b)$ is equivalent to $(\lap k)(z) \not = -1$ for
  all $z \in \C$ with $\Re z \ge 0$.
\end{thm}
Here $\lap k$ denotes the Laplace transform of $k$, i.e.
\begin{equation*}
  (\lap k)(z) = \int_0^\infty k(t) e^{-tz} \dd t
  = - \frac{K}{2} \int_0^\infty \hat{g}(t) e^{-tz} \dd t.
\end{equation*}
Hence the linear evolution of the order parameter is stable if
$(\lap k)(z) \not = -1$ for all $z \in \C$ with $\Re z \ge 0$. By
bounding the absolute value, we see that if $K < K_{ec}$ the system is
linearly stable. In case of the Gaussian or the Cauchy distribution,
$\hat{g}$ is always positive so that $K_{ec}$ equals the critical
coupling.

Like in the Vlasov equation \cite{penrose1960electrostatic}, this
condition can be related to a Penrose criterion for the complex
boundary, which also visualises how often the condition fails. Under
this we can also relate the stability to the known condition
\cite{strogatz1991stability,strogatz1992coupled}.

For small perturbations, our second main result is the nonlinear
stability of the incoherent state. In a first version we propagate
control in
\begin{equation*}
  \sup_{l \ge 1} \sup_{\xi \in \R} |u(t,l,\xi)| e^{a(\xi + tl/2)}.
\end{equation*}
\begin{thm}
  \label{thm:exponential-stability}
  Let $g \in \mathcal{M}(\R)$ and $a > 0$ such that $\hat{g}$
  satisfies the exponential stability condition of rate $a$ in
  \Cref{thm:order-parameter-stability}, i.e.
  $\hat{g} \in L^1(\R^+,\exp_a)$ and $(\lap k)(z) \not = -1$ for all
  $z \in \C$ with $\Re z \ge -a$. Then there exists a $\delta > 0$ and
  finite $c_1$ such that for initial data
  $\rho_{in} \in \mathcal{M}(\Gamma)$ with velocity marginal $g$,
  Fourier transform $u_{in}$ and
  $M_{in} := \sup_{l \in \N} \sup_{\xi \in \R} |u_{in}(l,\xi)|
  e^{a\xi} \le \delta$
  the Fourier transform $u$ of the solution to the Kuramoto equation
  with initial data $\rho_{in}$ satisfies
  \begin{equation*}
    \sup_{t \in \R^+} \sup_{l \ge 1} \sup_{\xi \in \R} |u(t,l,\xi)| e^{a(\xi+tl/2)}
    \le (1+c_1) M_{in}.
  \end{equation*}
  In particular, this shows that the order parameter $\eta(t) =
  u(t,1,0)$ decays as $O(e^{-a t})$.
\end{thm}

For the algebraic decay we propagate control in
\begin{equation*}
  \sup_{l \ge 1} \sup_{\xi \in \R^+}
  |u(t,l,\xi)| \frac{(1+\xi+t)^b}{(1+t)^{\alpha(l-1)}}.
\end{equation*}
\begin{thm}
  \label{thm:algebraic-stability}
  Let $g \in \mathcal{M}(\R)$ and $b > 1$ such that $\hat{g}$
  satisfies the algebraic stability condition of parameter $b$ in
  \Cref{thm:order-parameter-stability}, i.e.
  $\hat{g} \in L^1(\R^+,p_b)$ and $(\lap k)(z) \not = -1$ for all
  $z \in \C$ with $\Re z \ge 0$. For $\alpha > 0$ satisfying
  $b > 1 + \alpha$, there exists $\delta > 0$ and finite $\gamma_1$
  such that for initial data $\rho_{in} \in \mathcal{M}(\Gamma)$ with
  velocity marginal $g$, Fourier transform $u_{in}$ and
  $M_{in} := \sup_{l \in \N} \sup_{\xi \in \R} |u_{in}(l,\xi)|
  (1+\xi)^b \le \delta$
  the Fourier transform $u$ of the solution to the Kuramoto equation
  with initial data $\rho_{in}$ satisfies
  \begin{equation*}
    \sup_{t \in \R^+} \sup_{l \ge 1} \sup_{\xi \in \R^+}
    |u(t,l,\xi)| \frac{(1+\xi+t)^b}{(1+t)^{\alpha(l-1)}}
    \le (1+\gamma_1) M_{in}.
  \end{equation*}
  In particular, this shows that the order parameter $\eta(t) =
  u(t,1,0)$ decays as $O(t^{-b})$.
\end{thm}

Note that the exponential and algebraic control are one-sided and thus
do not control the $L^2$ or Sobolev norm in the original space. This
explains why we can have decay with these norms, even though there is
no decay of perturbations in the $L^2$ norm and the linear system is
only neutrally stable \cite{strogatz1991stability}. For the
application, however, the Sobolev norm $\mathcal{W}^{b,1}$ controls
$M_{in}$ and $g \in \mathcal{W}^{b,1}$ ensures the required decay of
$\hat{g}$.

If the stability condition fails, we can identify growing modes in the
Volterra equation.
\begin{thm}
  \label{thm:linear-volterra-modes}
  For $a \in \R$ suppose $\hat{g} \in L^1(\R^+,\exp_a)$ and that the
  convolution kernel $k$ defined by \Cref{eq:linear-order-kernel}
  satisfies $(\lap k)(z) \not = -1$ for $z \in \C$ with $\Re z =
  -a$.
  Then the solution to the Volterra
  equation~\eqref{eq:linear-order-volterra} is given by
  \begin{equation*}
    \eta(t) = u_{in}(1,t) - \int_0^\infty q(t-s) u_{in}(1,s) \dd s
  \end{equation*}
  with
  \begin{equation*}
    q(t) = r_s(t) + \sum_{i=1}^{n} \sum_{j=0}^{p_i-1} b_{i,j} t^j
    e^{\lambda_i t}
  \end{equation*}
  where $r_s$ is vanishing for negative arguments and satisfies $\|
  r_s \|_{L^1(\R^+,\exp_a)} < \infty$ and $\lambda_1,\dots ,\lambda_n$
  are the finitely many zeros of $1+ (\lap k)(z)$ in $\Re z > -a$ with
  multiplicities $p_1,\dots ,p_n$ and $b_{i,j}$ depends on the
  residues of $[1+(\lap k)(z)]^{-1}$ at $\lambda_i$.
\end{thm}

This says that the linear evolution is governed by eigenmodes $t^j
e^{\lambda_i t}$ for $i=1,\dots ,n$ and $j=0,\dots ,p_{i}-1$ and a
remaining stable part $r_s$. The previous linear stability theorem is
included in this formulation as $r=r_s$.

If $\hat{g} \in L^1(\R^+,\exp_a)$ for $a > 0$, this shows that if the
stability condition fails, there exists an unstable mode. Without
extra assumption, for every $a < 0$ we have
$\hat{g} \in L^1(\R^+,\exp_a)$ because $g$ is a probability
distribution. Hence, unless we are at the critical case, we have
unstable modes if the stability condition fails. Therefore, the linear
stability condition is sharp.

The additional restriction $(\lap k)(z) \not = -1$ for $\Re z = -a$ is
very weak, because, if $\hat{g} \in L^1(\R^+,\exp_a)$, then $(\lap
k)(z)$ is analytic for $\Re z > -a$ and by the Riemann-Lebesgue lemma
vanishing as $|z| \to \infty$. Hence we can choose a smaller $a'$ such
that $|a-a'|$ is arbitrarily small and the theorem applies.

The eigenmodes of the Volterra equation can be related to eigenmodes
$z_{\lambda,j}$ of the linear evolution (\Cref{eq:fourier-linear}) on
$C(\N \times \R)$. These eigenmodes are vanishing except in the first
spatial modes where they satisfy $z_{\lambda,j}(1,\cdot) \in
L^\infty(\R^+,\exp_a)$.

For the center-unstable manifold reduction we therefore look for
solutions in
\begin{equation*}
  \mathcal{Z}^a = \{ u \in C(\N \times \R):
  \| u \|_{\mathcal{Z}^a} < \infty \}
\end{equation*}
with
\begin{equation*}
  \| u \|_{\mathcal{Z}^a} = \sup_{l\ge 1}
  \|u(l,\cdot) \|_{L^\infty(\R,\exp_a)}
  = \sup_{l\ge 1} \sup_{\xi\in\R} e^{a\xi} |u(l,\xi)|.
\end{equation*}
The nonlinearity is not controlled in $\mathcal{Z}^a$ but within
\begin{equation*}
  \mathcal{Y}^a = \{ u \in C(\N \times \R):
  \| u \|_{\mathcal{Y}^a} < \infty \}
\end{equation*}
with
\begin{equation*}
  \| u \|_{\mathcal{Y}^a} = \sup_{l\ge 1} l^{-1}
  \|u(l,\cdot) \|_{L^\infty(\R,\exp_a)}
  = \sup_{l\ge 1} \sup_{\xi\in\R} \frac{e^{a\xi}}{l} |u(l,\xi)|.
\end{equation*}
By a spectral analysis we can find a continuous projection $P_{cu}$ from
$\mathcal{Y}^a$ to
$\mathcal{Z}^a_{cu} := \sspan{z_{\lambda_i,j} : i=1,\dots ,n \text{
    and } j=0,\dots,p_i-1}$
with complementary projection $P_s$ mapping $\mathcal{Z}^a$ to
$\mathcal{Z}^a_s$ and $\mathcal{Y}^a$ to $\mathcal{Y}^a_s$. The image
of $P_s$ is invariant under the linear evolution and decaying with
rate $a$. In fact the higher modes decay quicker in this norm, so that
the solution to the linear evolution with forcing in $\mathcal{Y}^a_s$
is within $\mathcal{Z}^a_s$.

Hence the theory of center-unstable manifold reduction
\cite{haragus2011local,vanderbauwhede1992center,vanderbauwhede1989centre}
applies and we can reduce the dynamics for the local behaviour. This
is our third main result and here we also consider $K$ as variable in
order to discuss the asymptotic result for couplings $K$ close to a
fixed coupling $K_c$.

\begin{thm}
  \label{thm:center-unstable-kuramoto}
  Given $a > 0$ and $g \in \mathcal{M}(\R)$ with
  $\hat{g} \in L^1(\R^+,\exp_a)$. Let $K_c$ be a coupling constant
  such that \Cref{thm:linear-volterra-modes} applies with
  center-unstable modes $\lambda_1,\dots ,\lambda_n$ satisfying
  $\Re \lambda_i \ge 0$ for $i=1,\dots,n$. For $k \in \N$ there exists
  $\psi \in C^k(\mathcal{Z}^a_{cu} \times \R, \mathcal{Z}^a_{s})$ and
  $\delta > 0$ such that for $|\epsilon| \le \delta$ the manifold
  \begin{equation*}
    \mathcal{M}_{\epsilon} = \{y + \psi(y,\epsilon) : y \in
    \mathcal{Z}^a_{cu}\}
  \end{equation*}
  is invariant and exponentially attractive under the nonlinear
  evolution with coupling constant $K = K_c + \epsilon$ in the region
  $\{ u \in \mathcal{Z}^a : \| u \|_{\mathcal{Z}^a} \le
  \delta\}$. Moreover, the derivatives of $\psi$ can explicitly be
  computed at $0$.
\end{thm}

Hence we can reduce the dynamics around the spatially homogeneous
distribution to the finite dimensional dynamics of
$\mathcal{M}_{\epsilon}$. We demonstrate the application by recovering
Chiba's bifurcation result for the Gaussian distribution
\cite{chiba2013ergodic}. In particular, it shows the nonlinear
stability of the partially locked states close to the bifurcation.

The key point for the center-unstable manifold reduction is the use of
the regularity in the natural frequency $\omega$. This allows the use
of the weighted space $\mathcal{Z}^a$ on which the linear generator
does not have a continuous spectrum along the imaginary axis as seen
in earlier linear analysis \cite{strogatz1991stability}. Instead the
linear evolution decays with rate $al$ in the $l$th spatial mode
except in the first mode $l=1$ where a discrete spectrum can
arise. Thus the effect of using $\mathcal{Z}^a$ is very similar to
adding white noise \cite{sakaguchi1988cooperative} of strength $D$ to
the system and in this case $D>0$ the stability and center manifold
reduction have been done
\cite{strogatz1991stability,bonilla1992nonlinear,crawford1994amplitude}.

We remark that the use of $\mathcal{Z}^a$ is compatible with adding
noise and thus for sufficiently regular velocity distribution the
center-unstable manifold reduction can be used to study the
bifurcation behaviour with the noise strength $D$ as additional
parameter in the limit $D \to 0$. The behaviour of the amplitude
equations in the center manifold reduction in this limit is reviewed
in \cite[Section 11]{strogatz2000from} and this explains why taking
the limit $D \to 0$ gave the correct behaviour for regular
distributions \cite{crawford1994amplitude}. Note that even in this
case the order parameter can decay faster than rate $D$ through the
Landau damping mechanism.

In order to understand the norm $\mathcal{Z}^a$ (which is also the
measure for the perturbations in \Cref{thm:exponential-stability}) we
relate the norm to the analytic continuation in the strip
$\{z \in \C: 0 \le \Im z \le a\}$ of the complex plane. Here we can
demonstrate explicitly that for suitable velocity marginals the
partially locked states have finite norm in $\mathcal{Z}^a$ and are
small perturbations if their order parameter $\eta$ is small. In
particular, this shows that partially locked states with a small $\eta$
and coupling $K$ just above $K_c$ result in exponential damping when
the coupling is lowered below $K_c$ by
\Cref{thm:exponential-stability}. This is particularly interesting,
because these partially locked states are not regular in general,
e.g. they do not even have a density.

The exponential stability of the incoherent state and the convergence
in $\mathcal{Z}^a$ to the reduced manifold do not imply convergence in
$L^2$. Nevertheless, we can relate it to weak convergence for
sufficiently nice test functions because the Fourier transform is
bounded.

In case of the algebraic stability (\Cref{thm:algebraic-stability}),
the convergence is very weak but we can conclude
$\int_0^{\infty} |\eta(t)| \dd t < \infty$. Going back to the original
equation this shows control in the gliding frame and weak convergence
of the phase marginal.

The plan of the paper is as follows: In \Cref{sec:mean-field-limit} we
review the notion of the considered solutions and their existence and
uniqueness. In \Cref{sec:stability-energy} we show the global energy
estimate (\Cref{thm:energy-stability}). The linear evolution and its
stability is discussed in \Cref{sec:linearised-stability}. In
\Cref{sec:nonlinear-stability} we use the linear stability to show
nonlinear stability. In preparation for the center-unstable manifold
reduction, the linear decomposition is derived in
\Cref{sec:spectral-decomposition}. With this the nonlinear reduction
(\Cref{thm:center-unstable-kuramoto}) is shown in
\Cref{sec:center-unstable-reduction}. The discussion of the norm and
the convergence are in \Cref{sec:norms}.

\subsection{Acknowledgements}
I would like to thank Clément Mouhot for many helpful suggestions and
discussions. Furthermore, I would like to thank Bastien Fernandez,
David Gérard-Varet and Giambattista Giacomin for inviting me to Paris
and discussing a first draft of the work.

\section{Mean-field limit}
\label{sec:mean-field-limit}

Early studies of the mean-field limit for the Kuramoto equation are
\cite{strogatz1991stability,sakaguchi1988phase}. Lancellotti
\cite{lancellotti2005vlasov} noted that the theory developed for the
Vlasov equation
\cite{braun1977vlasov,dobrushin1979vlasov,neunzert1984introduction}
applies because the interaction is Lipschitz. This setup is also
reviewed in \cite{carrillo2014contractivity} and was derived using
moments in \cite{chiba2013continuous}.

Since we want to be able to consider partially locked states, we
formulate solutions in terms of probability measures. As mentioned in
the introduction, we assume the initial data for time $t=0$ and solve
for $t \in \R^+$, because we have global existence and uniqueness.

\begin{definition}
  \label{def:kuramoto-solution}
  Let $C_{\mathcal{M}}$ be the solution space
  $C_{w}(\R^+,\mathcal{M}(\Gamma))$, which consists of the families of
  weakly continuous probability measures on $\Gamma$,
  i.e. $\rho_{\cdot} \in C_{\mathcal{M}}$ is a family
  $\{\rho_t \in \mathcal{M}(\R) : t \in \R^+\}$ such that for every
  $h \in C_b(\Gamma)$ the function $t \mapsto \rho_t(h)$ is
  continuous.

  Let $\rho_{in} \in \mathcal{M}(\Gamma)$. Then $\rho_{\cdot} \in
  C_{\mathcal{M}}$ is a solution to the Kuramoto equation with initial
  data $\rho_{in}$ if for all $h \in C^1_0(\R^+ \times \Gamma)$
  \begin{equation}
    \label{eq:kuramoto-weak}
    0 = \int_{\Gamma} h(0,P) \rho_{in}(\dd P)
    + \int_{\R^+} \int_{\Gamma}
    [ V[\rho_{\cdot}](t,P) \partial_{\theta} h(t,P)
    + \partial_t h(t,P) ] \rho_t(\dd P) \dd t,
  \end{equation}
  where
  \begin{equation*}
    V[\rho_{\cdot}](t,\theta,\omega) =
    \omega + \frac{K}{2i} \left( \eta(t) e^{-i\theta} - \conj{\eta(t)}
      e^{i\theta} \right)
    \text{ with }
    \eta(t) = \int_{(\theta,\omega) \in \Gamma} e^{i \theta} \rho_t(\dd
    \theta, \dd \omega).
  \end{equation*}
\end{definition}

The trajectories $T_{t,s}[\rho_{\cdot}] : \Gamma \mapsto \Gamma$ of
the system are defined as $T_{t,s}[\rho_{\cdot}](Q) = P(t)$ where $P :
\R^+ \mapsto \Gamma$ is the solution to the initial value problem
\begin{equation*}
  \begin{cases}
    \frac{\dd}{\dd t} P(t) = (V[\rho_{\cdot}](t,P(t)),0), \\
    P(s) = Q.
  \end{cases}
\end{equation*}
Since $|\eta(t)| \le 1$ for all $t \in \R^+$, the velocity field
$V[\rho_{\cdot}]$ is globally Lipschitz and the trajectories are
well-defined. Moreover, all derivatives of $V[\rho_{\cdot}](t,P)$ with
respect to $P$  are bounded, so that $T_{t,0}[\rho](P)$ also has
uniformly bounded derivatives of all orders for finite ranges of
$t$. Finally, $T_{t,s}$ is invertible with inverse $T_{s,t}$ and thus
a diffeomorphism.

By standard arguments for scalar transport equation, the following
lemma holds.
\begin{lemma}
  \label{thm:kuramoto-trajectories}
  Let $\rho_{in} \in \mathcal{M}(\Gamma)$. Then
  $\rho_{\cdot} \in C_{\mathcal{M}}$ is a solution to the Kuramoto
  equation if and only if
  $\rho_t = (T_{t,0}[\rho_{\cdot}])_{*}(\rho_{in})$, i.e.
  $\rho_t(A) = \rho_{in}((T_{t,0}[\rho_{\cdot}])^{-1}(A)) =
  \rho_{in}(T_{0,t}[\rho_{\cdot}](A))$
  for $A \in \mathcal{B}(\Gamma)$.
\end{lemma}

From this we see that the empirical measures of solutions of the
finite Kuramoto model are solutions to the Kuramoto equation in the
sense of \Cref{def:kuramoto-solution}. Neun\-zert's general framework
\cite{neunzert1984introduction,lancellotti2005vlasov} implies:
\begin{thm}
  The Kuramoto equation is well-posed. In particular, for every
  $\rho_{in} \in \mathcal{M}(\Gamma)$ there exists a unique global
  solution $\rho_{\cdot} \in C_{\mathcal{M}}$ to the Kuramoto equation
  with initial data $\rho_{in}$.
\end{thm}

The well-posedness implies that this is the mean-field limit as
$N \to \infty$. For this consider the empirical measure $\rho_{N,t}$
of a system of $N$ oscillators and suppose
$\rho_{N,0} \rightharpoonup \rho_{in}$. As the empirical measure is a
solution, the well-posedness shows that at later times $t$ also
$\rho_{N,t} \rightharpoonup \rho_{t}$.

\Cref{thm:kuramoto-trajectories} shows that the regularity is
preserved along the flow as we can see from the following lemma.
\begin{lemma}
  \label{thm:density-propagation}
  Let $\rho_{in} \in \mathcal{M}(\Gamma)$ and $\rho_{\cdot} \in
  C_{\mathcal{M}}$ be the solution to the Kuramoto equation with
  initial data $\rho_{in}$. If $\rho_{in}$ has a density $f_{in}$,
  then for every $t \in \R^+$ the measure $\rho_t$ has a density
  $f(t,\cdot)$ given by the flow
  \begin{equation*}
    f(t,P) = f_{in}(t,T_{0,t}[\rho_{\cdot}](P))
    \divergence(T_{0,t}[\rho_{\cdot}]).
  \end{equation*}
\end{lemma}
\begin{proof}
  If $\rho_{\cdot}$ is a solution,
  $\rho_t = (T_{t,0}[\rho_{\cdot}])_{*}(\rho_{in})$. Since
  $T_{t,0}[\rho_{\cdot}]$ is differentiable, $\rho_t$ has the density
  $f(t,\cdot)$.
\end{proof}

\Cref{thm:kuramoto-trajectories} also shows that the velocity marginal
is constant in time.
\begin{lemma}
  \label{thm:conservation-velocity-marginal}
  Let $\rho_{in} \in \mathcal{M}(\Gamma)$ and $g \in \mathcal{M}(\R)$
  be its velocity marginal. If $\rho_{\cdot} \in C_{\mathcal{M}}$ is
  the solution to the Kuramoto equation, then $g$ is for all $t \in
  \R^+$ the velocity marginal of $\rho_t$.
\end{lemma}
\begin{proof}
  By \Cref{thm:kuramoto-trajectories} $\rho_t =
  (T_{t,0}[\rho_{\cdot}])_{*}(\rho_{in})$ and the trajectories leave
  the velocity marginals invariant.
\end{proof}

For a solution $\rho_{\cdot} \in C_{\mathcal{M}}$ and velocity
distribution $g \in \mathcal{M}(\R)$ define the Fourier transform $u$
respectively $\hat{g}$ by \Cref{eq:fourier-u,eq:fourier-g}. The
evolution of the difference to the uniform state is captured by the
restriction to $l \ge 1$. Its evolution is described in Fourier
variables as follows:
\begin{thm}
  \label{thm:fourier-solution}
  Suppose $\rho_{in} \in \mathcal{M}(\Gamma)$ with velocity marginal
  $g \in \mathcal{M}(\R)$. Let $\rho_{\cdot} \in C_{\mathcal{M}}$ be
  the solution to the Kuramoto equation. Then the Fourier transform
  $u$ is in $C(\R^+ \times \N \times \R)$, is bounded uniformly by
  $1$ and satisfies
  \begin{equation}
    \label{eq:kuramoto-fourier-weak}
    \begin{split}
      0 &= \sum_{l\in\Z} \int_{\R} u_{in}(l,\xi) h(0,l,\xi) \dd
      \xi \\
      &+ \int_{\R^+} \int_{\R}
      \frac{K}{2} \hat{g}(\xi) \eta(t) h(t,1,\xi)
      + \sum_{l \ge 2} \frac{Kl}{2} u(t,l-1,\xi) \eta(t) h(t,l,\xi) \\
      &\quad + \sum_{l\in\N}
      \left[- \frac{Kl}{2} u(t,l+1,\xi) \conj{\eta(t)} h(t,l,\xi)
        + u(t,l,\xi)
        \left[\partial_t h(t,l,x) - l \partial_{\xi} h(t,l,\xi)\right]
      \right]
      \dd \xi \dd t
    \end{split}
  \end{equation}
  for all $h \in C^1_0(\R^+ \times \N \times \R)$, where
  $\eta(t) = u(t,1,0)$ and $\hat{g}$ is the Fourier transform of $g$.
\end{thm}
\begin{proof}
  For every $t \in \R^+$, the distribution $\rho_t$ is a probability
  measure, so that $u(t,l,\xi)$ is bounded by $1$ and continuous with
  respect to $\xi$. Since $\rho_{\cdot}$ is weakly continuous, $u$ is
  also continuous with respect to time.

  By \Cref{thm:conservation-velocity-marginal} the velocity marginal
  is always $g$ so that Plancherel's formula shows that
  \Cref{eq:kuramoto-fourier-weak} holds for all
  $h \in \mathscr{S}(\R^+ \times \N \times \R)$. By density it then
  holds for all $h \in C^1_0(\R^+ \times \N \times \R)$.
\end{proof}

The free transport is a well-posed linear problem, so that Duhamel's
principle on the free transport holds.
\begin{lemma}
  \label{thm:kuramoto-fourier-duhamel-free}
  Let $u_{in} \in C(\N \times \R)$, $\hat{g} \in C(\R)$ and
  $u \in C(\R^+ \times \N \times \R)$. For $t \in \R^+$ and
  $\xi \in \R$ define
  \begin{equation*}
    v(t,1,\xi) = u_{in}(1,\xi+t)
    + \frac{K}{2} \int_0^t
    \left[ \eta(s) \hat{g}(\xi+(t-s))
      - \conj{\eta(s)} u(s,2,\xi+(t-s)) \right ] \dd s
  \end{equation*}
  and for $l \ge 2$
  \begin{align*}
    v(t,l,\xi) &= u_{in}(l,\xi+lt) \\
    &+ \frac{Kl}{2} \int_0^t
    \left[ \eta(s) u(s,l-1,\xi+l(t-s))
      - \conj{\eta(s)} u(s,l+1,\xi+l(t-s)) \right ] \dd s
  \end{align*}
  where $\eta(t) = u(t,1,0)$. If $u$ satisfies
  \Cref{eq:kuramoto-fourier-weak} for all
  $h \in C^1_0(\R^+ \times \N \times \R)$, then $u \equiv v$.
\end{lemma}
\begin{proof}
  In the weak formulation \Cref{eq:kuramoto-fourier-weak} consider all
  terms except the free transport as forcing. Then the problem has a
  unique solution given by $v$.
\end{proof}

For the stability and bifurcation result we construct solutions $u$ to
\Cref{eq:kuramoto-fourier-weak}. By the following theorem these
correspond to the Fourier transform of the mean-field solution.

\begin{thm}
  \label{thm:fourier-uniqueness}
  Let $u_{in} \in C(\N \times \R)$ and $\hat{g} \in C(\R)$. Suppose
  $u, \tilde{u} \in C(\R^+ \times \N \times \R)$ satisfy
  \Cref{eq:kuramoto-fourier-weak} for every
  $h \in C^1_0(\R^+ \times \N \times \R)$. If for $T \in \R^+$ there
  exists $M \in \R^+$ with $\alpha,\beta \ge 0$ satisfying
  \begin{equation*}
    \sup_{\xi \in \R} |\hat{g}(\xi)| \min(1,e^{a\xi}) \le M
  \end{equation*}
  and
  \begin{equation*}
    \sup_{t \in [0,T]} \sup_{l \in N} \sup_{\xi \in \R}
    e^{-\beta l} \min(1,e^{a\xi})
    \max(|u(t,l,\xi)|, |\tilde{u}(t,l,\xi)|) \le M,
  \end{equation*}
  then $u(t,\cdot,\cdot) \equiv \tilde{u}(t,\cdot,\cdot)$ for
  $t \in [0,T]$.
\end{thm}
\begin{proof}
  For $\gamma > 0$ consider the difference measure $e$ defined by
  \begin{equation*}
    e(t) = \sup_{l\in\N} \sup_{\xi\in\R} e^{-\gamma t l}
    e^{-\beta l} \min(1,e^{a\xi})
    |u(t,l,\xi) - \tilde{u}(t,l,\xi)|.
  \end{equation*}
  Then by \Cref{thm:kuramoto-fourier-duhamel-free} for $t \in [0,T]$,
  $\l \in \N$ and $\xi \in \R$ we have
  \begin{align*}
    |u(t,l,\xi) - \tilde{u}(t,l,\xi)|
    e^{-\gamma t l} e^{-\beta l} \min(1,e^{a\xi})
    &\le 2KMl \int_0^t e^{2\beta+\gamma s} e(s) e^{-\gamma l(t-s)} \dd
      s \\
    &\le \frac{2KM}{\gamma} e^{2\beta + \gamma t}
    \sup_{s \in [0,t]} e(s).
  \end{align*}
  Choose $\gamma > 0$ and $t^* > 0$ so that
  $(2KM/\gamma) e^{2\beta + \gamma t^*} < 1/2$. Then the above shows
  for $t \in [0,t^*]$
  \begin{equation*}
    e(t) \le \frac{1}{2} \sup_{s \in [0,t^*]} e(s).
  \end{equation*}
  Therefore, we must have $e(t) = 0$ for $t \le t^*$, i.e. the
  solutions agree. Since we can repeat the argument in steps of $t^*$
  up to time $T$, this shows uniqueness up to time $T$.
\end{proof}

By \Cref{thm:fourier-solution} the Fourier transform $u$ of a solution
to the Kuramoto equation satisfies the assumptions of the previous
uniqueness theorem for every $\alpha$ and $\beta$.

\begin{thm}
  \label{thm:fourier-restriction}
  The restriction of the Fourier transform $u$ to
  $C(\R^+ \times \N \times \R^+)$, i.e. $\xi \ge 0$, satisfies an
  appropriate weak PDE (\Cref{thm:fourier-solution}) to which the
  Duhamel formula (\Cref{thm:kuramoto-fourier-duhamel-free}) applies
  and the uniqueness holds for the restriction $\xi \ge 0$ in
  \Cref{thm:fourier-uniqueness}.
\end{thm}
\begin{proof}
  Characteristics of the full equation are never entering the
  restricted region and the nonlinear interaction in the restricted
  region is determined by the region itself. Hence the restriction
  satisfies the appropriate restriction of the weak PDE
  \Cref{thm:fourier-solution} whose solution is given by the Duhamel
  formula \Cref{thm:kuramoto-fourier-duhamel-free}. The proof of the
  uniqueness theorem holds as before when restricting $\xi$ to $\R^+$
  everywhere.
\end{proof}

\section{Global stability by energy method}
\label{sec:stability-energy}

The basic idea is to differentiate the energy functional in time and
to note that the nonlinear terms cancel, because they are
skew-Hermitian. However, the Fourier transform $u$ does not need to
be differentiable. Therefore, we use an approximation procedure.

\begin{lemma}
  \label{thm:energy-differentiation}
  Let $\phi \in C^1(\R^+)$ be a bounded increasing weight with
  $\phi(0) = 1$ whose derivative $\phi'$ is bounded and satisfies
  $\phi'(\xi) > 0$ for $\xi \in \R^+$.

  For the Fourier transform $u \in C(\R^+ \times \N \times \R)$ of a
  solution to the Kuramoto equation define the functional $I(t)$ by
  \Cref{eq:energy-functional}. Then for all $t \in \R^+$
  \begin{equation*}
    I(t) \le I(0) + (\alpha-1) \int_0^t |\eta(s)|^2 \dd s
  \end{equation*}
  where
  \begin{equation*}
    \alpha =
      \frac{K^2}{4}
      \int_{\xi=0}^{\infty} \frac{|\hat{g}(\xi)|^2}{\phi'(\xi)}
      \phi^2(\xi) \dd \xi.
  \end{equation*}
\end{lemma}
\begin{proof}
  In order to overcome the non-differentiability, let
  $\chi \in C^\infty_c(\R)$ be a non-negative function with
  $\int_{x \in \R} \chi(x) \dd x = 1$ and let
  $\chi_{\delta}(x) = \delta^{-1}\chi(x/\delta)$ for $\delta >
  0$. Define the mollified $u_{\delta}$ by
  \begin{equation*}
    u_{\delta}(t,l,\xi) = \int_{x \in \R} \chi_{\delta}(\xi-x)
    u(t,l,x) \dd x.
  \end{equation*}
  Since $u$ is bounded by $1$, also $u_{\delta}$ is bounded by $1$ and
  has a bounded derivative with respect to $\xi$. Moreover,
  $\|u_{\delta}(t,l,\cdot)\|_2 \le \|u(t,l,\cdot)\|_2$.

  By \Cref{thm:kuramoto-fourier-duhamel-free}
  \begin{align*}
    u_{\delta}(t,1,\xi)
    &= u_{in,\delta}(1,\xi+l) \\
    &+ \frac{K}{2}
    \int_{0}^{t} \left[\eta(s) \hat{g}_{\delta}(\xi+t-s)
      - \conj{\eta(s)} u_{\delta}(s,2,\xi+t-s)\right] \dd s
  \end{align*}
  and for $l \ge 2$
  \begin{align*}
    u_{\delta}(t,l,\xi)
    &= u_{in,\delta}(l,\xi+tl) \\
    &+ \frac{Kl}{2}
    \int_{0}^{t} \left[\eta(s) u_{\delta}(s,l-1,\xi+(t-s)l)
      - \conj{\eta(s)} u_{\delta}(s,l+1,\xi+(t-s)l)\right] \dd s,
  \end{align*}
  where $\hat{g}_{\delta}$ and $u_{in,\delta}$ are the mollifications
  of $\hat{g}$ and $u_{in}$, respectively.  Hence $u_{\delta}$ is also
  continuously differentiable with respect to $t \in \R^+$ and
  satisfies classically
  \begin{equation*}
    \partial_t u_{\delta}(t,1,\xi)
    = \partial_{\xi} u_{\delta}(t,1,\xi)
    + \frac{K}{2} \left[ \eta(t) \hat{g}_{\delta}(\xi)
      - \conj{\eta(t)} u_{\delta}(t,2,\xi) \right]
  \end{equation*}
  and for $l \ge 2$
  \begin{equation*}
    \partial_t u_{\delta}(t,l,\xi)
    = l \partial_{\xi} u_{\delta}(t,l,\xi)
    + \frac{Kl}{2} \left[ \eta(t) u_{\delta}(t,l-1,\xi)
      - \conj{\eta(t)} u_{\delta}(t,l+1,\xi) \right].
  \end{equation*}
  For $\epsilon,\zeta > 0$ define
  \begin{equation*}
    I_{\delta,\epsilon,\zeta}(t) = \sum_{l=1}^{\infty}
    \int_{\xi=0}^{\infty} \frac{1}{l} |u_{\delta}(t,l,\xi)|^2
    \phi(\xi)
    e^{-\epsilon l} e^{-\zeta \xi} \dd \xi.
  \end{equation*}
  Since $u_{\delta}$ and $\partial_{\xi} u_{\delta}$ are bounded,
  $I_{\delta,\epsilon,\zeta}(t)$ is differentiable and we can
  differentiate under the integral sign to find
  \begin{align*}
    \partial_t I_{\delta,\epsilon,\zeta}(t)
    &= - \sum_{l=1}^{\infty}
      |u_{\delta}(t,l,0)|^2 e^{-\epsilon l}
      - \sum_{l=1}^{\infty} \int_{\xi=0}^{\infty}
      |u_{\delta}(t,l,\xi)|^2 \phi'(\xi)
      e^{-\epsilon l} e^{-\zeta \xi} \dd \xi \\
    &\quad + \zeta \sum_{l=1}^{\infty} \int_{\xi=0}^{\infty}
      |u_{\delta}(t,l,\xi)|^2 \phi(\xi)
      e^{-\epsilon l} e^{-\zeta \xi} \dd \xi \\
    &\quad + \int_{\xi=0}^{\infty}
      K \Re\left[\conj{\eta(t)} u_{\delta}(t,1,\xi) \conj{\hat{g}_{\delta}(\xi)}\right]
      e^{-\epsilon} e^{-\zeta \xi} \dd \xi \\
    &\quad + \sum_{l=1}^{\infty} \int_{\xi=0}^{\infty}
      K \Re\left[\eta(t)
      u_{\delta}(t,l,\xi)\conj{u_{\delta}(t,l+1,\xi)}\right]
      (e^{-\epsilon}-1) \phi(\xi) e^{-\epsilon l} e^{-\zeta \xi} \dd
      \xi \\
    &\le - |u_{\delta}(t,1,0)|^2 e^{-\epsilon}
      - \int_{\xi=0}^{\infty}
      |u_{\delta}(t,1,\xi)|^2 \phi'(\xi)
      e^{-\epsilon} e^{-\zeta \xi} \dd \xi\\
    &\quad + \int_{\xi=0}^{\infty}
      K \Re\left[\conj{\eta(t)} u_{\delta}(t,1,\xi) \conj{\hat{g}_{\delta}(\xi)}\right]
      e^{-\epsilon} e^{-\zeta \xi} \dd \xi \\
    &\quad + \left[\frac{\zeta}{\epsilon} + K |\eta(t)|
      \frac{(1-e^{\epsilon})}{\epsilon}\right] I_{\delta,\epsilon,\eta}(t)
  \end{align*}
  As $|\eta(t)| \le 1$, this shows by Gronwall's inequality for a
  fixed time $t$
  \begin{align*}
    &\exp\left(-\left[\frac{\zeta}{\epsilon} + K
      \frac{(1-e^{\epsilon})}{\epsilon}\right]t\right)
      I_{\delta,\epsilon,\eta}(t) \\
    &\le I_{\delta,\epsilon,\eta}(0)
      - e^{-\epsilon} \int_0^t
      |u_{\delta}(s,1,0)|^2 \dd s \\
    &\quad + e^{-\epsilon} \int_0^t \int_{\xi=0}^{\infty} \left(
      -|u_{\delta}(s,1,\xi)|^2 \phi'(\xi)
      +
      K \Re\left[\conj{\eta(s)} u_{\delta}(s,1,\xi)
      \conj{\hat{g}_{\delta}(\xi)}\right]
      \right) e^{-\zeta \xi} \dd \xi \dd s
  \end{align*}
  Taking $\delta \to 0$, we find by dominated convergence
  \begin{align*}
    &\exp\left(-\left[\frac{\zeta}{\epsilon} + K
      \frac{(1-e^{\epsilon})}{\epsilon}\right]t\right)
      I_{\epsilon,\eta}(t) \\
    &\le I_{\epsilon,\eta}(0)
      - e^{-\epsilon} \int_0^t
      |\eta(s)|^2 \dd s \\
    &\quad + e^{-\epsilon} \int_0^t \int_{\xi=0}^{\infty} \left(
      -|u(s,1,\xi)|^2 \phi'(\xi)
      +
      K \Re\left[\conj{\eta(s)} u(s,1,\xi)
      \conj{\hat{g}(\xi)}\right]
      \right) e^{-\zeta \xi} \dd \xi \dd s,
  \end{align*}
  where
  \begin{equation*}
    I_{\epsilon,\zeta}(t) = \sum_{l=1}^{\infty}
    \int_{\xi=0}^{\infty} \frac{1}{l} |u(t,l,\xi)|^2
    \phi(\xi)
    e^{-\epsilon l} e^{-\zeta \xi} \dd \xi.
  \end{equation*}
  The integral can be controlled as
  \begin{align*}
    \int_{\xi=0}^{\infty} &\left(
      -|u(s,1,\xi)|^2 \phi'(\xi)
      +
      K \Re\left[\conj{\eta(s)} u(s,1,\xi)
      \conj{\hat{g}(\xi)}\right]
      \right) e^{-\zeta \xi} \dd \xi \\
    &= - \int_{\xi=0}^{\infty}
      \left|u(s,1,\xi) - \frac{K\conj{\eta(s)}}{2}
      \frac{\conj{\hat{g}(\xi)}}{\phi'(\xi)} \phi(\xi) \right|^2 \phi'(\xi)
      e^{-\zeta\xi}
      \dd \xi \\
    &\qquad + |\eta(s)|^2 \frac{K^2}{4}
      \int_{\xi=0}^{\infty} \frac{|\hat{g}(\xi)|^2}{\phi'(\xi)}
      \phi^2(\xi) e^{-\zeta\xi} \dd \xi \\
    &\le \alpha |\eta(s)|^2.
  \end{align*}
  Hence, the previous bound shows
  \begin{equation*}
    \exp\left(-\left[\frac{\zeta}{\epsilon} + K
      \frac{(1-e^{\epsilon})}{\epsilon}\right]t\right)
      I_{\epsilon,\eta}(t)
    \le I_{\epsilon,\eta}(0)
      + (\alpha-1) e^{-\epsilon} \int_0^t
      |\eta(s)|^2 \dd s.
  \end{equation*}
  Now taking $\epsilon\to 0$ with $\zeta = \epsilon^2$ shows by
  monotone convergence the claimed result.
\end{proof}

The global stability result, \Cref{thm:energy-stability}, follows from
this lemma by finding a suitable weight $\phi$ such that $\alpha <
1$. For this we want to minimise the integral defining
$\alpha$. Formally, the Euler-Lagrange equation is
\begin{equation*}
  0 = |\hat{g}(\xi)| + \partial_\xi \left[ |\hat{g}(\xi)|
    \left(\frac{\phi(\xi)}{\phi'(\xi)}\right) \right].
\end{equation*}
With $\phi(0) = 1$ the solution formally is
\begin{equation*}
  \phi(\xi)
  = \frac{A}{A - \int_0^\xi |\hat{g}(\zeta)| \dd \zeta}
\end{equation*}
for a constant $A$. If $A > \int_0^\infty |\hat{g}(\xi)| \dd \xi$, this
indeed defines a bounded increasing weight and we find
\begin{equation*}
  \frac{K^2}{4} \int_0^\infty \frac{|\hat{g}(\xi)|^2}{\phi'(\xi)}
  \phi^2(\xi) \dd \xi
  = \frac{K^2}{4} A \int_0^\infty |\hat{g}(\xi)| \dd \xi.
\end{equation*}
Thus if $K < K_{ec}$, for a suitable $A$ we have the required control.

\begin{proof}[Proof of \Cref{thm:energy-stability}]
  \Cref{thm:energy-differentiation} requires the additional condition
  $\phi'(\xi) > 0$ for all $\xi \in \R^+$. Therefore, we slightly
  modify the weight found formally before.

  For $A > \int_0^\infty |\hat{g}(\xi)| \dd \xi$ and $\gamma > 0$ the
  integral
  \begin{equation*}
    \frac{K^2}{4} \int_0^\infty
    A \frac{|\hat{g}(\xi)|^2}{|\hat{g}(\xi)|+e^{-\gamma\xi}} \dd \xi
  \end{equation*}
  tends to
  \begin{equation*}
    \frac{K^2}{4}
    \left(\int_0^\infty |\hat{g}(\xi)| \dd \xi \right)^2
  \end{equation*}
  as $A \to \int_0^\infty |\hat{g}(\xi)| \dd \xi$ and
  $\gamma \to \infty$. Since $K < K_{ec}$, we can thus find
  $\bar{\gamma} > 0$ and
  $\bar{A} > \bar{\gamma}^{-1} + \int_0^\infty |\hat{g}(\xi)| \dd \xi$
  such that
  \begin{equation*}
    \frac{K^2}{4} \int_0^\infty
    \bar{A} \frac{|\hat{g}(\xi)|^2}{|\hat{g}(\xi)|+e^{-\bar{\gamma}\xi}} \dd \xi
    < 1.
  \end{equation*}
  With $\bar{A}$ and $\bar{\gamma}$ define the weight $\phi$ by
  \begin{equation*}
    \phi(\xi) = \frac{\bar{A}}{\bar{A} - \int_{0}^{\xi} \left(|\hat{g}(\zeta)| +
        e^{-\bar{\gamma}\zeta}\right) \dd \zeta}.
  \end{equation*}

  The integrand is continuous so that $\phi \in C^1(\R^+)$ and
  $\phi'(\xi) > 0$ for all $\xi \in \R^+$. Moreover, $\phi$ is
  bounded and $\phi(0)=1$.

  Thus $\phi$ is a valid weight for the theorem and by
  \Cref{thm:energy-differentiation} a solution satisfies the required
  bound
  \begin{equation*}
    I(t) + c \int_0^t |\eta(s)|^2 \dd s \le I(0)
  \end{equation*}
  with
  \begin{equation*}
    c = 1 - \alpha = 1 - \frac{K^2}{4} \int_0^\infty
    \bar{A} \frac{|\hat{g}(\xi)|^2}{|\hat{g}(\xi)|+e^{-\bar{\gamma}\xi}} \dd \xi
    > 0.
  \end{equation*}
  For the initial bound, use Plancherel to note
  \begin{equation*}
    I(0) \le \| \phi \|_{\infty} \sum_{l\ge1} \int_{\xi\in\R^+}
    |u_{in}(l,\xi)|^2 \dd \xi \le \| \phi \|_{\infty} \| f \|_2^2.
    \qedhere
  \end{equation*}
\end{proof}

\section{Stability of the linearised evolution}
\label{sec:linearised-stability}

As in the mean-field limit, the linear evolution given by
\Cref{eq:fourier-linear} is understood as PDE for continuous solutions
in the same weak sense, i.e. tested against $C^1_0$ functions.
\begin{proof}[Proof of \Cref{thm:linear-evolution}]
  The free transport has a unique weak solution, so that, as in
  \Cref{thm:kuramoto-fourier-duhamel-free}, Duhamel's formula implies
  \begin{equation}
    \label{eq:proof-linear-evolution-duhamel}
    (e^{tL} u_{in})(1,\xi) = u_{in}(1,\xi+t)
    + \frac{K}{2} \int_0^t
    (e^{sL} u_{in})(1,0) \hat{g}(\xi+(t-s)) \dd s
  \end{equation}
  and for $l \ge 2$
  \begin{equation*}
    (e^{tL} u_{in})(l,\xi) = u_{in}(l,\xi+lt).
  \end{equation*}
  Taking $\xi = 0$ in \Cref{eq:proof-linear-evolution-duhamel} shows
  that $\nu(t) = (e^{tL} u_{in})(1,0)$ must satisfy
  \Cref{eq:linear-order-volterra}. By the general theory of Volterra
  equation \cite[Chapter 2]{gripenberg1990volterra}
  \Cref{eq:linear-order-volterra} has a unique solution and the
  solution is continuous. Hence the evolution operator $e^{tL}$ is
  well-defined and given by
  \Cref{eq:linear-order-volterra,eq:linear-order-first-mode,eq:linear-order-higher-modes}.
\end{proof}
The solution of the Volterra equation can be characterised by the
following lemma.
\begin{lemma}
  For a convolution kernel $k \in L^1_{loc}(\R^+)$ there exists a
  unique resolvent kernel $r \in L^1_{loc}(\R^+)$ satisfying $r + r
  \conv k = k$. For $f \in L^1_{loc}(\R^+)$ the unique solution $x \in
  L^1_{loc}(\R^+)$ to the Volterra equation $x + k \conv x = f$ is
  given by $x = f - r \conv f$.
\end{lemma}
\begin{proof}
  This is Theorem 3.1 and 3.5 of Chapter 2 of
  \cite{gripenberg1990volterra}.
\end{proof}
We can control the convolution by using the (sub)multiplicative
structure of the weights.
\begin{lemma}
  Let $r,f \in L^1_{loc}(\R^+)$. For $a \in \R$
  \begin{equation*}
    \| r \conv f \|_{L^\infty(\R^+,\exp_a)}
    \le \| r \|_{L^1(\R^+,\exp_a)}
    \| f \|_{L^\infty(\R^+,\exp_a)}
  \end{equation*}
  and for $A \ge 1$ and $b \ge 0$
  \begin{equation*}
    \| r \conv f \|_{L^\infty(\R^+,p_{A,b})}
    \le \| r \|_{L^1(\R^+,p_{b})}
    \| f \|_{L^\infty(\R^+,p_{A,b})}.
  \end{equation*}
\end{lemma}
\begin{proof}
  For $a \in \R$ and $t \in \R^+$
  \begin{align*}
    |(r \conv f)(t)| e^{at}
    &\le \int_0^t |r(t-s)| e^{a(t-s)} |f(s)| e^{as} \dd s \\
    &\le \| f \|_{L^\infty(\R^+,\exp_a)}
      \int_0^t |r(t-s)| e^{a(t-s)} \dd s
  \end{align*}
  showing the first bound.

  For $A \ge 1$ and $b \ge 0$ and $t \in \R^+$
  \begin{align*}
    |(r \conv f)(t)| (A+t)^b
    &\le \int_0^t |r(t-s)| (1+t-s)^b |f(s)| (A+s)^b \dd s \\
    &\le \| f \|_{L^\infty(\R^+,p_{A,b})}
      \int_0^t |r(t-s)| (1+t-s)^b \dd s
  \end{align*}
  where we used that $A + t \le (1+t-s)(A+s)$ holds for $s,t \in
  \R^+$, because $A \ge 1$. Taking the supremum shows the second
  claim.
\end{proof}

These norms propagate in the linear evolution.
\begin{lemma}
  \label{thm:linear-stability-first-mode}
  Let $u_{in} \in C(\N \times \R)$. Then for $a \in \R$
  \begin{equation*}
    \sup_{t \ge 0} e^{at}
    \| (e^{tL} u_{in})(1,\cdot) \|_{L^\infty(\R^+, \exp_a)}
    \le c_1 \| u_{in}(1,\cdot) \|_{L^\infty(\R^+, \exp_a)}
  \end{equation*}
  with
  $c_1 = 1 + (1 + \| r \|_{L^1(\R^+,\exp_a)}) (K/2) \| g
  \|_{L^1(\R^+,\exp_a)}$ and for $A \ge 1$ and $b \ge 0$
  \begin{equation*}
    \sup_{t \ge 0} \sup_{\xi \ge 0} |(e^{tL}u_{in})(1,\xi)|
    (A+\xi+t)^b
    \le \gamma_1 \| u_{in}(1,\cdot) \|_{L^\infty(\R^+,p_{A,b})}
  \end{equation*}
  with
  $\gamma_1 = 1 + (1 + \| r \|_{L^1(\R^+,p_b)}) (K/2) \| g
  \|_{L^1(\R^+,p_b)}$.
\end{lemma}
\begin{proof}
  By the previous lemma and the explicit formula for $\nu$,
  \Cref{eq:linear-order-first-mode} shows for $\xi \in \R$
  \begin{align*}
    | &(e^{tL} u_{in})(1,\xi) | e^{a(t+\xi)}
    \le \| u_{in} \|_{L^{\infty}(\R^+,\exp_a)} \\
    &+ \int_0^t \frac{K}{2} (1 + \| r \|_{L^1(\R^+,\exp_a)})
    \| u_{in} \|_{L^{\infty}(\R^+,\exp_a)} e^{a(\xi+t-s)}
    |\hat{g}(\xi+t-s)| \dd s
  \end{align*}
  which proves the claim.

  In the algebraic case we find for $\xi \ge 0$
  \begin{align*}
    | &(e^{tL} u_{in})(1,\xi) | (A+\xi+t)^{b}
    \le \| u_{in} \|_{L^{\infty}(\R^+,p_{A,b})} \\
    &+ \int_0^t \frac{K}{2} (1 + \| r \|_{L^1(\R^+,p_b)})
    \| u_{in} \|_{L^{\infty}(\R^+,p_{A,b})} \frac{(A+\xi+t)^b}{(A+s)^b}
    |\hat{g}(\xi+t-s)| \dd s
  \end{align*}
  which implies the result as $(A+\xi+t)/(A+s) \le (1+\xi+t-s)$.
\end{proof}

The stability condition on the kernel comes from the literature.
\begin{proof}[Proof of \Cref{thm:order-parameter-stability}]
  The exponential condition follows from Lemma 3.4 and Theorem 4.1 of
  Section 2 of \cite{gripenberg1990volterra} and is called the
  Paley-Wiener theorem. The algebraic condition comes from Gel'fand's
  theorem as Theorem 4.3 in Section 4 in
  \cite{gripenberg1990volterra}, because the weight $p_b$ is
  submultiplicative.
\end{proof}

The stability criterion $1+\lap k(z) \not = 0$ for $\Re z \ge 0$ can
be understood as imposing that $\lap \hat{g}(\{z : \Re z \ge 0\})$
does not contain $2/K$. Since $\lap \hat{g}$ is analytic, the argument
principle shows that this is equivalent to imposing that $\lap
\hat{g}(i\R)$ does not encircle $2/K$.

For $\Re z > 0$ we find by Fubini
\begin{equation*}
  \lap \hat{g}(z)
  = \int_0^\infty e^{-z\xi} \hat{g}(\xi) \dd \xi
  = \int_{-\infty}^{\infty} g(\omega) \frac{1}{z-i\omega} \dd \omega.
\end{equation*}

By continuity of $\lap \hat{g}$ we thus find
\begin{equation*}
  \lap \hat{g}(ix)
  = \lim_{\lambda \to 0+0} \int_{-\infty}^{\infty}
  \frac{g(\omega)}{i(x-\omega)+\lambda} \dd \omega
  = i \principalV \int_{-\infty}^{\infty}
  \frac{g(x+\omega)}{\omega} \dd \omega
  + \pi g(x),
\end{equation*}
where the last equality is Plemelj formula for continuous $g$ and
$\principalV$ denotes the principal value integral.

The curve $\lap \hat{g}(i\R)$ starts and ends at $0$, is bounded, and
goes through the right half plane where it crosses the real axis by
continuity of the principal value integral. If $g > 0$, this shows
that it crosses the real axis at a positive value. Thus there exists a
critical value $K_c$ such that for $0\le K < K_c$ the system in
linearly stable and for $K_c + \delta > K > K_c$ with some
$\delta > 0$ the system is unstable.

For the Gaussian distribution, we can plot the boundary explicitly
(cf. \Cref{fig:gaussianStability}) and note that for $K < 2/(\pi
g(0))$ the solution is stable while for $K > 2/(\pi g(0))$ there is
exactly one root of $(1+\lap k)(z)$ for $\Re z \ge 0$.

\begin{figure}[htb]
  \setlength\fwidth{\textwidth}
  \setlength\fheight{6cm}
  \input{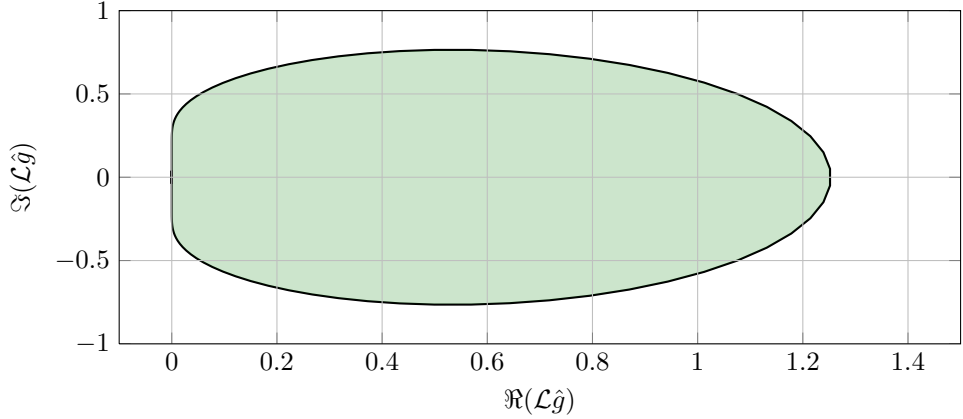}
  \caption{The black line shows $\lap \hat{g}(i \R)$ with the enclosed
  green area $\lap \hat{g}(\{z : \Re z > 0\})$ for the Gaussian
  distribution $g(\omega) = (2\pi)^{-1/2} e^{-\omega^2/2}$.}
  \label{fig:gaussianStability}
\end{figure}

Already for two Gaussians with variance $1$ centred at
$\omega = \pm 1.5$, we see a more interesting behaviour
(cf. \Cref{fig:gaussianTwoStability}). Here we see that at the
critical coupling two roots of $(1+\lap k)(z)$ appear and that for
sufficiently large $K$ this reduces to one root.

\begin{figure}[htb]
  \setlength\fwidth{\textwidth}
  \setlength\fheight{6cm}
  \input{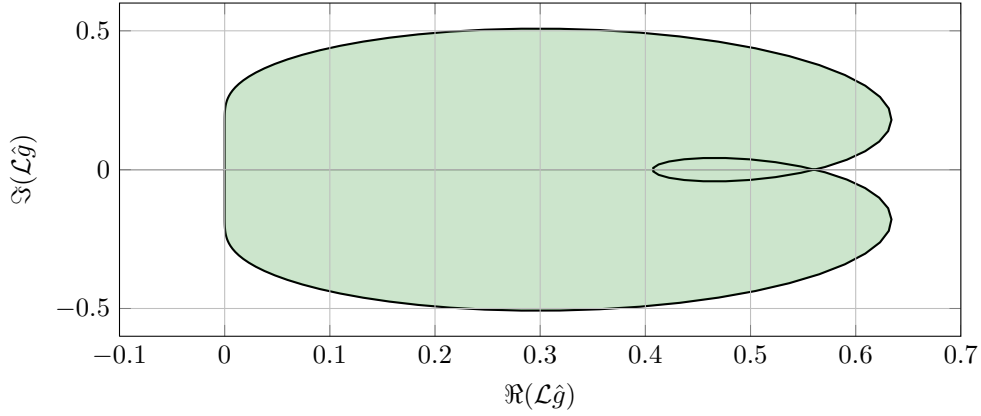}
  \caption{The black line shows $\lap \hat{g}(i \R)$ with the enclosed
  green area $\lap \hat{g}(\{z : \Re z > 0\})$ for two Gaussian
  distributions with variance $1$ and centred at $\omega = \pm 1.5$.}
  \label{fig:gaussianTwoStability}
\end{figure}

A sufficient condition for stability is
\begin{equation*}
  K < \frac{2}{\pi \|g\|_{\infty}}.
\end{equation*}
In this case $\Re(\lap \hat{g}(i x)) < 2/K$ so that $\lap \hat{g}(i
\R)$ cannot encircle $2/K$ and the solution is stable. In the case of
a symmetric distribution with a maximum at $0$, this condition is
sharp for the first instability because then $(\lap \hat{g})(0) =
\pi g(0)$.

Comparing the linear stability condition to possible stationary
solutions, we can understand the stability condition on the real part
$\pi g(x)$ as critical mass density to form an instability while the
imaginary part ensures that it is not drifted away.

\section{Nonlinear stability}
\label{sec:nonlinear-stability}

In Fourier variables the nonlinear interaction is given by
$R:C(\N \times \R) \mapsto C(\N \times \R)$ defined as
\begin{equation*}
  \begin{cases}
    R(v)(1,\xi) = - \frac{K}{2} \conj{v(1,0)} v(2,\xi), \\
    R(v)(l,\xi) = \frac{Kl}{2} \left[ v(1,0) v(l-1,\xi) - \conj{v(1,0)}
      v(l+1,\xi) \right] \text{ for $l \ge 2$}
  \end{cases}
\end{equation*}
for $ v \in C(\N \times \R)$. Here we see that the nonlinearity is not
bounded as it increases with the spatial mode $l$. In the exponential
convergence we can match this with the faster decay of the linear
operator in higher modes. In the algebraic case we compensate for this
in the norm by the factor $(1+t)^{-\alpha(l-1)}$.

We characterise the solution by Duhamel's principle on the linear
evolution.
\begin{lemma}
  \label{thm:duhamel-linear}
  Let $\hat{g} \in C(\R)$ and $e^{tL}$ the linear evolution from
  \Cref{thm:linear-evolution}. For $u_{in} \in C(\N \times \R)$ define
  the map $T : C(\R^+ \times \N \times \R) \mapsto C(\R^+ \times \N
  \times \R)$ by
  \begin{equation*}
    (Tu)(t,l,\xi) = (e^{tL}u_{in})(l,\xi)
    + \int_0^t (e^{(t-s)L} R(u(s,\cdot,\cdot)))(l,\xi) \dd s.
  \end{equation*}
  Then $u \in C(\R^+ \times \N \times \R)$ satisfies
  \Cref{eq:kuramoto-fourier-weak} for all $h \in C^1_0(\R^+ \times \N
  \times \R)$ if and only if $T u = u$.
\end{lemma}
\begin{proof}
  For $u$ consider the nonlinear terms as forcing in
  \Cref{eq:kuramoto-fourier-weak}. The linear problem has a unique
  solution given by $Tu$. Hence $u$ is a solution if and only if
  $Tu = u$.
\end{proof}

For the exponential and algebraic stability we want to propagate the
norms given in the introduction using
\Cref{thm:duhamel-linear}. However, we do not know a priori if these
norms stay finite. Therefore, we construct a solution using Banach
fixed point theorem and by the uniqueness
(\Cref{thm:fourier-uniqueness}) this must be the mean-field solution.
\begin{proof}[Proof of \Cref{thm:exponential-stability}]
  Define on $C(\R^+ \times \N \times \R)$ the norm $\| \cdot \|_{e}$ by
  \begin{equation*}
    \| u \|_{e} =
    \sup_{t \in \R^+} \sup_{l \ge 1} \sup_{\xi \in \R} |u(t,l,\xi)| e^{a(\xi+tl/2)}
  \end{equation*}
  for $u \in C(\R^+ \times \N \times \R)$.

  By the assumption \Cref{thm:linear-stability-first-mode} applies
  with a finite constant $c_1$. This shows for
  $u \in C(\R^+ \times \N \times \R)$ and $t \in \R^+, \xi \in \R$
  \begin{align*}
    e^{a(\xi+t/2)} |Tu(t,1,\xi)|
    &\le c_1 e^{-at/2} M_{in}
    + \frac{K c_1 \| u \|_{e}^2}{2} \int_0^t e^{-a(t-s)/2} \dd s \\
    &\le c_1 M_{in} + \frac{Kc_1}{a} \| u \|_{e}^2
  \end{align*}
  and by the explicit form (\Cref{thm:linear-evolution}) for $l \ge 2$
  \begin{align*}
    e^{a(\xi+tl/2)} |Tu(t,l,\xi)|
    &\le e^{-atl/2} M_{in}
    + K l \| u \|_{e}^2 \int_0^t e^{-al(t-s)/2} \dd s \\
    &\le M_{in} + \frac{2K}{a} \| u \|_{e}^2.
  \end{align*}
  Hence
  \begin{equation*}
    \| Tu \|_{e} \le c_1 M_{in} + c_2 \| u \|_{e}^2
  \end{equation*}
  where $c_2 = \max(c_1,2) K / a$. Therefore, there exists $\delta >
  0$ such that if $M_{in} \le \delta$ and $\| u \|_{e} \le (1+c_1)
  M_{in}$, then $\|Tu\|_e \le (1+c_1) M_{in}$.

  If $u,\tilde{u} \in C(\R^+ \times \N \times \R)$ satisfy
  $\|u\|_{e}, \|\tilde{u}\|_{e} \le (1+c_1)\delta$, then as before we
  find
  \begin{equation*}
    \| Tu - T\tilde{u} \|_{e}
    \le \frac{2K(1+c_1)\delta}{a} \max(c_1,2) \|u - \tilde{u}\|_{e}.
  \end{equation*}
  Hence we can choose $\delta$ small enough such that $T$ is a
  contraction on $\{u \in C(\R^+ \times \N \times \R) : \| u \|_{e}
  \le (1+c_1) M_{in} \}$ for $M_{in} \le \delta$.

  Thus for $M_{in} \le \delta$, there exists a unique fixed point $u$
  with $\| u \|_{e} \le (1+c_1) M_{in}$. By
  \Cref{thm:fourier-uniqueness}, this $u$ must be equal to the Fourier
  transform of the solution to the Kuramoto equation.
\end{proof}

The structure for the proof of algebraic stability is very similar to
the previous proof of exponential estimate.
\begin{proof}[Proof of \Cref{thm:algebraic-stability}]
  Define on $C(\R^+ \times \N \times \R)$ the norm $\| \cdot \|_{a}$
  by
  \begin{equation*}
    \| u \|_{a} =
    \sup_{t\in\R^+} \sup_{l \ge 1} \sup_{\xi \in \R^+}
    |u(t,l,\xi)| \frac{(1+\xi+t)^b}{(1+t)^{\alpha(l-1)}}.
  \end{equation*}
  for $u \in C(\R^+ \times \N \times \R)$.

  By the assumption \Cref{thm:linear-stability-first-mode} applies
  with a finite constant $\gamma_1$. This shows for
  $u \in C(\R^+ \times \N \times \R)$ and $t\in\R^+, \xi \in \R$
  \begin{align*}
    |Tu(t,1,\xi)| (1+\xi+t)^b
    &\le \gamma_1 M_{in} + \frac{K\|u\|_a^2}{2} \int_0^t
    \gamma_1 (1+s)^{\alpha-b} \dd s \\
    &\le \gamma_1 M_{in} + \frac{\gamma_1 K}{2(b-1-\alpha)} \| u \|_{a}^2.
  \end{align*}
  For modes $l \ge 2$ the explicit formula \Cref{thm:linear-evolution}
  shows
  \begin{equation*}
    |Tu(t,l,\xi)| \frac{(1+\xi+t)^b}{(1+t)^{\alpha(l-1)}}
    \le M_{in} + K \|u\|^2_a l \int_0^t
    \frac{(1+s)^{\alpha l-b}}{(1+t)^{\alpha(l-1)}} \dd s.
  \end{equation*}
  For $l \in \N$ with $\alpha l > b - 1$ we can bound
  \begin{equation*}
    l \int_0^t \frac{(1+s)^{\alpha l-b}}{(1+t)^{\alpha(l-1)}} \dd s
    \le \frac{l}{\alpha l-b+1}
  \end{equation*}
  by a finite constant independent of $l$.

  For $l \in \N$ with $l \ge 2$ and $\alpha l - b \le 1$ the integral
  $\int_0^t (1+s)^{\alpha l-b} \dd s$ grows at most logarithmically,
  so that it can be bounded by $(1+t)^{\alpha(l-1)}$. Hence there
  exists a finite $\gamma_2$ satisfying
  \begin{equation*}
    \| Tu \|_{a} \le \gamma_1 M_{in} + \gamma_2 \| u \|_a^2.
  \end{equation*}

  Thus there exists $\delta > 0$ such that $M_{in} \le \delta$ implies
  that $\| T u \|_a \le (1+\gamma_1) M_{in}$ holds if $\| u \|_{a} \le
  (1+\gamma_1) M_{in}$.

  For two functions $u,\tilde{u} \in C(\R^+ \times \N \times \R)$ with
  $\| u \|_a, \| \tilde{u} \|_{a} \le (1+\gamma_1) \delta$, we find
  for $t\in\R^+$ and $\xi \in \R^+$ as before
  \begin{align*}
    |(Tu-T\tilde{u})(t,1,\xi)| (1+\xi+t)^b
    &\le K \gamma_1 (1+\gamma_1) \delta \|u - \tilde{u}\|_{a}
    \int_0^t (1+s)^{\alpha-b} \dd s \\
    &\le \frac{K \gamma_1 (1+\gamma_1) \delta}{b-1-\alpha} \|u -
    \tilde{u}\|_{a}
  \end{align*}
  and for $l \ge 2$
  \begin{equation*}
    |(Tu - T\tilde{u})(t,l,\xi)|
    \frac{(1+\xi+t)^b}{(1+t)^{\alpha(l-1)}}
    \le 2K (1+\gamma_1) \delta \| u - \tilde{u}\|_a
    l \int_0^t \frac{(1+s)^{\alpha l-b}}{(1+t)^{\alpha(l-1)}} \dd s.
  \end{equation*}
  By the same argument, the factors with $l$ are uniformly bounded
  over $l \ge 2$, so that for small enough $\delta > 0$ the map $T$ is
  a contraction on
  $\{u \in C(\R^+ \times \N \times \R) : \| u \|_{a} \le (1+\gamma_1)
  M_{in} \}$
  when $M_{in} \le \delta$. Hence there exists a unique fixed point
  $Tu = u$ with $\| u \|_{a} \le (1+\gamma_1) M_{in}$. For every
  $T > 0$, the restricted uniqueness (\Cref{thm:fourier-uniqueness}
  and \Cref{thm:fourier-restriction}) shows that this solution must be
  equal to the Fourier transform of the mean-field solution to the
  Kuramoto equation up to time $T$. Hence $u$ always equals the
  Fourier transform of the solution and the bound on $u$ proves the
  theorem.
\end{proof}

\section{Linear eigenmodes and spectral decomposition}
\label{sec:spectral-decomposition}

The solution of the Volterra equation with eigenmodes follows from the
literature.
\begin{proof}[Proof of \Cref{thm:linear-volterra-modes}]
  By the assumed decay of $\hat{g}$, the Laplace transform $\lap k$ is
  analytic in $\{z : \Re z > -a\}$ and continuous up to the boundary
  $\{z : \Re z = -a\}$. By the Riemann-Lebesgue lemma further $\lap
  k(z) \to 0$ as $|z| \to \infty$ for $\Re z > -a$. Hence $1 + \lap
  k(z)$ can only have finitely many poles with finite multiplicity.

  The decomposition now follows from Theorem 2.1 of Chapter 7 of
  \cite{gripenberg1990volterra}.
\end{proof}

For the characterisation of the stable subspace we use the functional
$\alpha$.
\begin{lemma}
  For $a \in \R$ and $\lambda \in \C$ and $j \in \N$ with $\Re \lambda
  > -a$ define the continuous functional $\alpha_{\lambda,j}$ on
  $\mathcal{Y}^a$ by
  \begin{equation*}
    \alpha_{\lambda,j}(u)
    = \left.\left(\frac{\dd}{\dd z}\right)^j
      [\lap u(1,\cdot)](z) \right|_{z=\lambda}
    = \int_0^\infty u(1,t) (-t)^j e^{-\lambda t} \dd t.
  \end{equation*}
\end{lemma}
\begin{proof}
  Since $|u(1,\xi)| \le e^{-a\xi} \| u \|_{\mathcal{Y}^a}$, the
  functional is well-defined, continuous and for $j \ge 1$ the
  derivative can be computed by differentiating under the integral
  sign.
\end{proof}

\begin{lemma}
  \label{thm:stable-subspace}
  Assume $a \in \R$ and $\hat{g} \in C(\R)$ as in
  \Cref{thm:linear-volterra-modes} and let $\lambda_1,\dots ,\lambda_n$ be the
  roots of multiplicity $p_1,\dots ,p_n$ as in
  \Cref{thm:linear-volterra-modes}. Define the closed subspace
  \begin{equation*}
    \mathcal{Y}^a_s = \{ u \in \mathcal{Y}^a
    : \alpha_{\lambda_i,j}(u) = 0 \text{ for $i=1,\dots,n$ and
      $j=0,\dots ,p_i-1$}\}.
  \end{equation*}
  Then $\mathcal{Y}^a_s$ is invariant under the linear evolution
  $e^{tL}$ and $u \in \mathcal{Y}^a$ is in $\mathcal{Y}^a_s$ if and
  only if for $\nu(t) = (e^{tL}u)(1,0)$ holds $\nu \in
  L^\infty(\R^+,\exp_a)$.
  In this case
  \begin{equation}
    \label{eq:linear-stable-part-control}
    \sup_{t \ge 0} e^{at}
    \| (e^{tL} u)(1,\cdot) \|_{L^\infty(\R^+, \exp_a)}
    \le c_1 \| u(1,\cdot) \|_{L^\infty(\R^+, \exp_a)}
  \end{equation}
  holds with
  $c_1 = 1 + (1 + \| r_s \|_{L^1(\R^+,\exp_a)}) (K/2) \| g
  \|_{L^1(\R^+,\exp_a)} < \infty$.
\end{lemma}
\begin{proof}
  If $u \in \mathcal{Y}^{a}_s$, the solution of the Volterra equation
  for $\nu$ in \Cref{thm:linear-volterra-modes} simplifies to
  \begin{equation*}
    \nu(t) = u(1,t) - (r_s \conv u(1,\cdot))(t).
  \end{equation*}
  Since $r_s \in L^1(\R, \exp_a)$, this shows
  $\nu \in L^\infty(\R^+,\exp_a)$ and, as in
  \Cref{thm:linear-stability-first-mode},
  \Cref{eq:linear-stable-part-control} follows with the given constant
  $c_1$.

  Conversely, assume $\nu \in L^\infty(\R^+,\exp_a)$. Then at $z \in
  \C$ with $\Re z > -a$ the Laplace transform of $\nu$, $k$ and
  $u(1,\cdot)$ are analytic and by the Volterra equation
  \eqref{eq:linear-order-volterra} satisfy
  \begin{equation*}
    (\lap \nu)(z) + (\lap k)(z) (\lap \nu)(z)
    = [\lap u(1,\cdot)](z).
  \end{equation*}
  Hence the LHS has roots $\lambda_1,\dots ,\lambda_n$ of
  multiplicities $p_1,\dots ,p_n$. Thus the RHS must have the same
  roots, i.e. $\alpha_{\lambda_i,j}(u) = 0$ for $i=1,\dots,n$ and
  $j=0,\dots,p_i-1$, which shows $u \in \mathcal{Y}^a_s$.

  The growth condition $\nu \in L^\infty(\R^+,\exp_a)$ under the
  linear evolution is invariant, so that $\mathcal{Y}^a_s$ is
  invariant under the linear evolution.
\end{proof}

We can use this to control the decay in
$\mathcal{Z}^a_s = \mathcal{Y}^a_s \cap \mathcal{Z}^a$ under the
linear evolution and the forcing term from a Duhamel equation.
\begin{lemma}
  \label{thm:convolution-integral}
  Assume the hypothesis of \Cref{thm:stable-subspace} with $a \ge
  0$. Then for $u \in \mathcal{Z}^a_s$ we have for $t \in \R^+$
  \begin{equation*}
    \| e^{tL} u \|_{\mathcal{Z}^a}
    \le c_1 e^{-at} \| u \|_{\mathcal{Z}^a}.
  \end{equation*}

  For $0 \le \mu < a$ consider $F \in C(\R^+ \times \N \times \R)$
  with $F(t,\cdot,\cdot) \in \mathcal{Y}^a_s$ for all $t \in \R^+$ and
  norm
  \begin{equation*}
    \| F \|_{\mathcal{Y}^a,-\mu}
    := \sup_{t\in\R^+} e^{-\mu t} \| F(t,\cdot,\cdot)
    \|_{\mathcal{Y}^a}
    < \infty.
  \end{equation*}
  Then
  \begin{equation*}
    v = \int_0^\infty e^{tL} F(t,\cdot,\cdot) \dd t
  \end{equation*}
  is well-defined and $v \in \mathcal{Z}^a_s$ with
  \begin{equation*}
    \| v \|_{\mathcal{Z}^a} \le c(\mu) \| F \|_{\mathcal{Y}^a,-\mu}
  \end{equation*}
  where $c(\mu) = (a-\mu)^{-1} c_1$.
\end{lemma}
\begin{remark}
  The forcing $F$ is measured in the weaker norm $\mathcal{Y}^a$ and
  we are able to recover the control in the stronger norm
  $\mathcal{Z}^a$.
\end{remark}
\begin{proof}
  Since $u$ is in $\mathcal{Z}^a_s$, we find by
  \Cref{thm:stable-subspace} for $t \in \R^+$ and $\xi \in \R$
  \begin{equation*}
    e^{a\xi} |(e^{tL}u)(1,\xi)|
    \le c_1 e^{-at} \| u \|_{\mathcal{Z}^a}
  \end{equation*}
  and by \eqref{eq:linear-order-higher-modes} for $l \ge 2$
  \begin{equation*}
    e^{a\xi} | (e^{tL}u)(l,\xi) |
    \le e^{-alt} \| u \|_{\mathcal{Z}^a}
  \end{equation*}
  which shows the first claim.

  For the second part note that for every $t \in \R^+$ the integrand
  $e^{tL}F(t,\cdot,\cdot)$ is continuous. By
  \Cref{thm:stable-subspace} we can control for $\xi \in \R$
  \begin{equation*}
    e^{a \xi}
    |[e^{tL}F(t,\cdot,\cdot)](1,\xi)|
    \le c_1 e^{-at} \| F(t,\cdot,\cdot) \|_{\mathcal{Y}^a}
    \le c_1 e^{(\mu-a)t} \| F \|_{\mathcal{Y}^a,\mu}
  \end{equation*}
  and for $l \ge 2$ by \Cref{eq:linear-order-higher-modes}
  \begin{equation*}
    e^{a\xi}
    |[e^{tL}F(t,\cdot,\cdot)](l,\xi)|
    \le l e^{(\mu-al)t} \| F \|_{\mathcal{Y}^a,\mu}.
  \end{equation*}
  These bounds are uniformly integrable as
  \begin{equation*}
    \int_0^\infty c_1 e^{(\mu-a)t} \dd t = \frac{c_1}{a-\mu}
  \end{equation*}
  and for $l \ge 2$
  \begin{equation*}
    l \int_0^\infty e^{(\mu-al)t} \dd t = \frac{l}{al-\mu}
    = \frac{1}{a-\mu/l}.
  \end{equation*}
  Hence the integral is well-defined and defines a continuous function
  $v \in C(\N \times \R)$. Moreover, it shows the claimed control of
  $\|v\|_{\mathcal{Z}^a}$.
\end{proof}

In order to find the eigenmodes in $\mathcal{Z}^a$ and the projection
to them, we consider the linear generator $L$ on $\mathcal{Z}^a$.
\begin{lemma}
  Assume the hypothesis of \Cref{thm:stable-subspace} with $a >
  0$. Define the closed operator $L : D(L) \subset \mathcal{Z}^a
  \mapsto \mathcal{Y}^a$ with $D(L) = \{u \in C^1(\N \times \R) : \| u
  \|_{\mathcal{Z}^a} < \infty \}$ by
  \begin{equation*}
    (Lu)(1,\xi) = \partial_{\xi} u(1,\xi)
    + \frac{K}{2} u(1,0) \hat{g}(\xi)
  \end{equation*}
  and for $l \ge 2$ by
  \begin{equation*}
    (Lu)(1,\xi) = l \partial_{\xi} u(l,\xi).
  \end{equation*}
  Then in $\{ \lambda : \Re \lambda > -a\}$ the spectrum of $L$ equals
  $\sigma = \{ \lambda_1,\dots ,\lambda_n\}$. For $\lambda \not \in
  \sigma$ and $\Re \lambda > -a$ the resolvent $(L-\lambda)^{-1}$ is
  given by
  \begin{equation}
    \label{eq:resolvent-first-mode}
    [(L-\lambda)^{-1}v](1,\xi)
    = -
    \int_{\xi}^{\infty} e^{-\lambda(\zeta-\xi)}
    \left(v(1,\zeta) + \frac{K}{2} \;
    \frac{[\lap v(1,\cdot)](\lambda)}{1+(\lap k)(\lambda)} \;
    \hat{g}(\zeta) \right)\dd \zeta
  \end{equation}
  for $\xi \in \R$ where $k$ is the convolution kernel from
  \Cref{eq:linear-order-kernel} and for $l \ge 2$
  \begin{equation}
    \label{eq:resolvent-higher-modes}
    [(L-\lambda)^{-1}v](l,\xi)
    = - \frac{1}{l} \int_{\xi}^{\infty} e^{-(\lambda/l)(\zeta-\xi)}
      v(l,\zeta) \dd \zeta.
  \end{equation}
\end{lemma}
\begin{proof}
  By classical analysis the differentiation and the uniform limit can
  be interchanged so that $L$ is closed.

  For the resolvent fix $\lambda \in \C$ with $\Re \lambda > -a$ and
  $\lambda \not \in \sigma$ and $v \in \mathcal{Y}^a$. We now look for
  solutions of $(L-\lambda)u=v$ in $u \in D(L) \subset \mathcal{Z}^a$.

  For $l \ge 2$ this implies
  \begin{equation*}
    l \partial_{\xi} u(l,\xi) - \lambda u(l,\xi)
    = v(l,\xi)
    \; \Rightarrow \;
    \partial_{\xi} \left( e^{-(\lambda/l)\xi} u(l,\xi) \right)
    = \frac{1}{l} v(l,\xi) e^{-(\lambda/l)\xi}.
  \end{equation*}
  Since $\Re \lambda > -a$ this shows that $u$ is uniquely given by
  \Cref{eq:resolvent-higher-modes}.

  For $l=1$ we likewise find that $u$ must satisfy
  \begin{equation*}
    u(1,\xi) = - \int_{\xi}^{\infty} e^{-\lambda(\zeta-\xi)}
    \left( v(1,\zeta) - \frac{K}{2} u(1,0) \hat{g}(\zeta) \right)
    \dd \zeta.
  \end{equation*}
  Hence by taking $\xi = 0$ we find that $u$ must satisfy
  \begin{equation*}
    u(1,0) \left(1 + (\lap k)(\lambda)\right) =
    - [\lap v(1,\cdot)](\lambda).
  \end{equation*}
  Since $\lambda \not \in \sigma$, this determines $u(1,0)$ uniquely
  and the solution is given by
  \Cref{eq:resolvent-first-mode}. Moreover, the found solution
  satisfies $\| u \|_{\mathcal{Z}^a} < \infty$. Hence such a $\lambda$
  is not in the spectrum of $L$ and the resolvent map takes the given
  form.

  If $\lambda \in \sigma$, the above shows that $L$ is not injective,
  so that $\sigma$ is in the spectrum.
\end{proof}

\Cref{eq:resolvent-first-mode} shows that at a root $\lambda$ of
multiplicity $p$ the residue of the resolvent map $(L-\lambda)^{-1}$
is in the space spanned by $\sspan{z_{\lambda,j} : j=0,\dots ,p-1}$
where for $\xi \in \R$
\begin{equation*}
  z_{\lambda,j}(1,\xi) = \frac{K}{2} \int_{\xi}^{\infty}
  e^{-\lambda(\zeta-\xi)} (\zeta-\xi)^j \hat{g}(\zeta) \dd \zeta
\end{equation*}
and for $l \ge 2$
\begin{equation*}
  z_{\lambda,j}(l,\xi) = 0.
\end{equation*}
Since we assume $\hat{g} \in L^1(\R,\exp_a)$, we always have
$z_{\lambda,j} \in D(L) \subset \mathcal{Z}^a$.
\begin{lemma}
  \label{thm:generalised-eigenvectors}
  Assume the hypothesis of \Cref{thm:stable-subspace} with $a >
  0$. Then for a root $\lambda$ with $\Re \lambda > -a$ of $1+(\lap
  k)(z)$ with multiplicity $p$, the elements $z_{\lambda,0},\dots
  ,z_{\lambda,p-1}$ are generalised eigenvectors of $L$ with
  eigenvalue $\lambda$.
\end{lemma}
\begin{proof}
  Note that for $j=0,\dots,p-1$
  \begin{equation*}
    z_{\lambda,j}(1,0) = - \left.\left(\frac{\dd}{\dd z}\right)^j
      (\lap k)(z)\right|_{z=\lambda}
    =
    \begin{cases}
      1 & j=0, \\
      0 & \text{otherwise}.
    \end{cases}
  \end{equation*}
  Then by direct computation
  \begin{equation*}
    L z_{\lambda,j}
    =
    \begin{cases}
      \lambda z_{\lambda,j} & j=0, \\
      \lambda z_{\lambda,j} - j z_{\lambda,j-1} & \text{otherwise}.
    \end{cases}
    \qedhere
  \end{equation*}
\end{proof}

Using holomorphic functional calculus, we can construct the required
projection.
\begin{lemma}
  \label{thm:linear-projections}
  Assume the hypothesis of \Cref{thm:stable-subspace} with $a >
  0$. There exists a continuous projection $P_{cu} : \mathcal{Y}^a
  \mapsto \mathcal{Z}^a_{cu}$ such that the complementary projection
  $P_s = I - P_{cu}$ maps to $\mathcal{Y}^a_s$.
\end{lemma}
\begin{proof}
  Take $\gamma$ a contour in $\{z \in \C : \Re z > -a\}$ encircling
  all roots $\lambda_1,\dots ,\lambda_n$ once and define
  \begin{equation*}
    P_{cu}(u)
    = \frac{-1}{2\pi i} \oint_{\gamma} (L-\lambda)^{-1}(u) \dd \lambda.
  \end{equation*}
  Then $P_{cu}$ is a continuous projection. Since its value is given
  by the residues, \Cref{eq:resolvent-first-mode} shows that $P_{cu}$
  maps into $\mathcal{Z}^a_{cu}$.

  On the image of the complementary projection $P_s$ the resolvent map
  $(L-\lambda)^{-1}$ is continuous for $\Re \lambda > -a$. By
  \Cref{eq:resolvent-first-mode} this can only be true if for all
  roots $\lambda_1,\dots ,\lambda_n$ with multiplicities $p_1,\dots
  ,p_n$ also $\lap u(1,\cdot)$ has a root of matching multiplicity,
  i.e. $\alpha_{\lambda_i,j}$ must be vanishing for $i=1,\dots,n$ and
  $j=0,\dots,p_i-1$. Hence the image is within $\mathcal{Y}^a_s$.
\end{proof}

\section{Center-unstable manifold reduction}
\label{sec:center-unstable-reduction}

Since the nonlinearity $R$ is smooth from $\mathcal{Z}^a$ to
$\mathcal{Y}^a$, the center-unstable manifold reduction follows from
the decay on the stable part (\Cref{thm:convolution-integral}) and the
continuous projections (\Cref{thm:linear-projections}). The
center-manifold reduction in Banach spaces is discussed in
\cite[Chapter 2]{vanderbauwhede1992center}, which can be adapted to
our case with a weak PDE as Duhamel's principle holds again. Moreover,
instead of a center manifold reduction we do a center-unstable
manifold reduction \cite[Section 1.5]{vanderbauwhede1989centre} as
this implies convergence to the reduced manifold. The statements are
also discussed in \cite{haragus2011local}.

In order to understand the bifurcation behaviour, the center-unstable
manifold is constructed with $\epsilon = K-K_c$ as additional
parameter. Furthermore, we need to localise the non-linearity around
the incoherent state. This localisation might affect the resulting
reduced manifold, which is not unique. Moreover, we choose the
localisation to preserve the rotation symmetry of the problem so that
the resulting manifold has the same symmetry.

Hence we study the solution in
$\tilde{\mathcal{Z}}^a = \mathcal{Z}^a \times \R$ and
$\tilde{\mathcal{Y}}^a = \mathcal{Y}^a \times \R$ with norms
$\|(u,\epsilon)\|_{\tilde{\mathcal{Z}}^a} = \max(\| u
\|_{\mathcal{Z}^a},|\epsilon|)$
and
$\|(u,\epsilon)\|_{\tilde{\mathcal{Y}}^a} = \max(\| u
\|_{\mathcal{Y}^a},|\epsilon|)$.
In these define the linear evolution operator $e^{t\tilde{L}}$ as
\begin{equation*}
  e^{t\tilde{L}} \tilde{u}
  = e^{t\tilde{L}} (u,\epsilon)
  = (e^{tL} u, \epsilon).
\end{equation*}
The center-unstable space $\tilde{Z}^a_{cu}$ is now spanned by $(0,1)$
and $(z_{\lambda_i,j},0)$ for $i=1,\dots,n$ and $j=0,\dots
,p_i-1$.
The projection $\tilde{P}_{cu}$ is defined as
$\tilde{P}_{cu}(u,\epsilon) = (P_{cu}u,\epsilon)$ with $P_{cu}$ from
\Cref{thm:linear-projections}. The complementary projection
$\tilde{P}_s$ maps into the stable part
$\tilde{\mathcal{Y}}^a_s = \mathcal{Y}^a_s \times \{ 0 \}$.

For the center-unstable manifold reduction the used properties of the
linear evolution are collected in the following lemma.
\begin{lemma}
  \label{thm:center-unstable-linear}
  Assume the hypothesis of \Cref{thm:center-unstable-kuramoto}. Then
  there exists a continuous projection
  $\tilde{P}_{cu} : \mathcal{Y}^a \mapsto \tilde{\mathcal{Z}}^a_{cu}$
  with complementary projection $\tilde{P}_s = I - \tilde{P}_{cu}$
  mapping $\tilde{\mathcal{Y}}^a$ to $\tilde{\mathcal{Y}}^a_{s}$. The
  subspaces $\tilde{\mathcal{Z}}^a_{cu}$ and $\tilde{\mathcal{Y}}^a_s$
  are invariant under the linear evolution. On
  $\tilde{\mathcal{Z}}^a_{cu}$ the linear evolution acts as
  finite-dimensional matrix exponential whose generator has spectrum
  $\lambda_1,\dots ,\lambda_n$. On
  $\tilde{\mathcal{Z}}^a_{s} = \tilde{P}_s \tilde{\mathcal{Z}}^a$
  there exists a constant $c_1$ such that for $t \in \R^+$ and
  $u \in \tilde{\mathcal{Z}}^a_s$
  \begin{equation*}
    \| e^{t\tilde{L}} u \|_{\tilde{\mathcal{Z}}^a}
    \le c_1 e^{-at} \| u \|_{\tilde{\mathcal{Z}}^a}
  \end{equation*}
  and there exists a continuous function $c : [0,a) \mapsto \R^+$
  such that for $\mu \in [0,a)$ and
  $F \in C(\R^+ \times \N \times \R \times \R)$ with
  $F(t,\cdot) \in \tilde{\mathcal{Y}}^a_s$ for $t \in \R^+$ and norm
  \begin{equation*}
    \| F \|_{\mathcal{Y}^a,-\mu}
    := \sup_{t\in\R^+} e^{-\mu t} \| F(t,\cdot)
    \|_{\tilde{\mathcal{Y}}^a}
    < \infty,
  \end{equation*}
  the integral
  \begin{equation*}
    v = \int_0^\infty e^{t\tilde{L}} F(t,\cdot) \dd t
  \end{equation*}
  is well-defined and $v \in \tilde{\mathcal{Z}}^a_s$ with
  \begin{equation*}
    \| v \|_{\tilde{\mathcal{Z}}^a} \le c(\mu) \| F \|_{\tilde{\mathcal{Y}}^a,-\mu}.
  \end{equation*}
\end{lemma}
\begin{proof}
  The linear evolution on $\epsilon$ is constant so that the
  properties follow from the study of $e^{tL}$,
  i.e. \Cref{thm:convolution-integral,thm:generalised-eigenvectors,thm:linear-projections}.
\end{proof}

In the extended dynamics for $\tilde{u}$ in $\tilde{\mathcal{Z}}^a$,
the nonlinearity takes the form
$\tilde{R} : \tilde{\mathcal{Z}}^a \mapsto \tilde{Y}^a$ given by
\begin{equation*}
  \tilde{R}(\tilde{u}) = (N(\tilde{u}),0)
\end{equation*}
where $N : \tilde{\mathcal{Z}}^a \mapsto \mathcal{Y}^{a}$ is for
$\tilde{u} = (u,\epsilon)$ defined by
\begin{equation*}
  N(\tilde{u})(1,\xi) =
  \frac{\epsilon}{2} u(1,0) \hat{g}(\xi)
  - \frac{(K_c + \epsilon)}{2} \conj{u(1,0)} u(2,\xi)
\end{equation*}
for $\xi \in \R$ and for $l \le 2$
\begin{equation*}
  N(\tilde{u})(l,\xi) =
  \frac{(K_c + \epsilon)l}{2}
  \left(u(1,0) u(l-1,\xi) - \conj{u(1,0)} u(l+1,\xi) \right).
\end{equation*}

For the localisation let $\chi : \R^+ \mapsto [0,1]$ be a smooth
function with $\chi(x) = 1$ for $x \le 1$ and $\chi(x) = 0$ for
$x \ge 2$. Then let $\chi_{\delta}(x) = \chi(x/\delta)$ and
$s_\delta(x) = x \chi_{\delta}(|x|)$. With this define the localised
nonlinearity
$\tilde{R}_{\delta} : \tilde{\mathcal{Z}}^a \mapsto \tilde{\mathcal{Y}}^a$ by
\begin{equation*}
  \tilde{R}_{\delta}(\tilde{u}) = (N_{\delta}(\tilde{u}),0)
\end{equation*}
where $N_{\delta} : \tilde{\mathcal{Z}}^a \mapsto \mathcal{Y}^{a}$ is for
$\tilde{u} = (u,\epsilon)$ defined by
\begin{equation*}
  N_{\delta}(\tilde{u})(1,\xi) =
  \frac{s_{\delta}(\epsilon)}{2} s_{\delta}(u(1,0)) \hat{g}(\xi)
  - \frac{(K_c + s_{\delta}(\epsilon))}{2}
  s_{\delta}(\conj{u(1,0)}) s_{\delta}(u(2,\xi))
\end{equation*}
for $\xi \in \R$ and for $l\ge 2$
\begin{equation*}
  N_{\delta}(\tilde{u})(l,\xi) =
  \frac{(K_c + s_{\delta}(\epsilon))l}{2}
  \left(s_{\delta}(u(1,0)) s_{\delta}(u(l-1,\xi))
    - s_{\delta}(\conj{u(1,0)}) s_{\delta}(u(l+1,\xi)) \right).
\end{equation*}

Then for $\delta > 0$, the localised nonlinearity $\tilde{R}_{\delta}$
agrees with the original nonlinearity of the mean-field equation for
$\| \tilde{u} \|_{\tilde{\mathcal{Z}}^a} \le \delta$, is in
$C^\infty(\tilde{\mathcal{Z}}^a, \tilde{\mathcal{Y}}^a)$ with
$\tilde{R}_{\delta}(0) = D \tilde{R}_{\delta}(0) = 0$ and is bounded
and Lipschitz continuous with constants tending to $0$ as
$\delta \to 0$.

We now study the evolution PDE
\begin{equation}
  \label{eq:localised-evolution}
  \begin{cases}
    \partial_t u(t,1,\xi) = \partial_{\xi} u(t,1,\xi)
    + \frac{K_c}{2} u(t,1,0) \hat{g}(\xi)
    + N_{\delta}(\tilde{u})(1,\xi) ,\\
    \partial_t u(t,l,\xi) = l \partial_{\xi} u(t,l,\xi)
    + N_{\delta}(\tilde{u})(l,\xi)
    & \text{for $l\ge 2$}, \\
    \partial_t \epsilon(t) = 0,
  \end{cases}
\end{equation}
for
$\tilde{u} = (u,\epsilon) \in C(\Delta \times \N \times \R \times \R)$
understood weakly against test functions in $C^1_0$ and where
$\Delta \subset \R$ is a time interval.

\begin{lemma}
  A solution $\tilde{u}$ of \Cref{eq:localised-evolution} satisfies
  Duhamel's formula, i.e. for $t_0,t_1 \in \Delta$ with $t_0 \le t_1$
  the following holds
  \begin{equation*}
    \tilde{u}(t_1) = e^{(t_1-t_0)\tilde{L}} \tilde{u}(t_0)
    + \int_{t_0}^{t_1} e^{(t_1-s)\tilde{L}} \tilde{R}_{\delta}(\tilde{u}(s))
    \dd s.
  \end{equation*}
\end{lemma}
\begin{proof}
  Consider $\tilde{R}_{\delta}(\tilde{u})$ as forcing. Then the linear
  problem has a unique solution given by the formula.
\end{proof}

\begin{lemma}
  Let $\tilde{u}$ be a solution of
  \Cref{eq:localised-evolution}. Given
  $\tilde{u}(t_0) \in \tilde{\mathcal{Z}}^a$ the solution is uniquely
  determined at later times $t \ge t_0$ and
  $\| \tilde{u}(t) \|_{\tilde{\mathcal{Z}}^a}$ is uniformly bounded
  for compact intervals of later times. If $\delta$ is small enough,
  there exists, given $\tilde{u}(t_0) \in \tilde{\mathcal{Z}}^a$, a
  global solution for $t \ge t_0$.
\end{lemma}
\begin{proof}
  The projections commute with the linear evolution in the Duhamel
  formula and are bounded. Since $\tilde{R}_{\delta}$ is bounded, this
  shows with the control of the linear evolution
  (\Cref{thm:center-unstable-linear}) that
  $\| \tilde{u}(t) \|_{\tilde{\mathcal{Z}}^a}$ is uniformly bounded
  for compact time-intervals of later times. Furthermore,
  $\tilde{R}_{\delta}$ is Lipschitz so that the Duhamel formulation
  gives a Gronwall control showing uniqueness.

  If $\delta$ is small enough, then the Picard iteration
  \begin{equation*}
    \tilde{u}
    \mapsto e^{(t-t_0)\tilde{L}} \tilde{u}(t_0) + \int_{t_0}^{t}
    e^{(t-s)\tilde{L}} \tilde{R}_{\delta}(\tilde{u}(s)) \dd s
  \end{equation*}
  is a contraction, so that a solution exists.
\end{proof}

With this we can construct the reduced manifold. We want to capture
the evolution in the reduced manifold where the stable part has
decayed, which we can characterise by the growth as $t \to - \infty$.

\begin{thm}
  \label{thm:center-unstable-localised}
  Assume the hypothesis of \Cref{thm:center-unstable-kuramoto}. There
  exists $\mu \in (0,a)$ and $\delta_b$ such that for
  $\delta \le \delta_b$ there exists
  $\tilde{\psi} \in C^k(\tilde{\mathcal{Z}}^a_{cu},\tilde{\mathcal{Z}}^a_s)$
  such that
  \begin{align*}
    \mathcal{M} &=
    \{ \tilde{u}_0 \in \tilde{\mathcal{Z}}^a
    : \exists \tilde{u} \text{ satisfying
    \Cref{eq:localised-evolution} with $\tilde{u}(0) = \tilde{u}_0$} \\
    &\qquad \qquad \qquad \text{and $\sup_{t\in\R^-} e^{\mu t} \| \tilde{u}(t)
    \|_{\tilde{\mathcal{Z}}^a} < \infty$} \} \\
    &= \{ v + \tilde{\psi}(v) : v \in \tilde{\mathcal{Z}}^a_{cu} \}
  \end{align*}
  which is the reduced manifold. Under the evolution of
  \Cref{eq:localised-evolution} the reduced manifold is invariant and
  exponentially attractive. Moreover, the Taylor series of
  $\tilde{\psi}$ is explicitly computable at $0$ starting with
  \begin{align*}
    &\tilde{\psi}|_0 = 0, \\
    &D\tilde{\psi}|_0 = 0, \\
    &D^2\tilde{\psi}|_0(h_1,h_2) = \int_{-\infty}^{0} e^{-t\tilde{L}} \tilde{P}_s
    D^2 \tilde{R}_{\delta}|_0(e^{t\tilde{L}} h_1, e^{t\tilde{L}} h_2) \dd t.
  \end{align*}
\end{thm}

The growth condition already suggest that this captures the dominant
behaviour and $\mathcal{M}$ is attractive.

\begin{proof}
  Since we have a Duhamel formula, the proofs of Theorem 1 and 2 of
  \cite{vanderbauwhede1992center} work. We further adapt it, as in
  Section 1.5 of \cite{vanderbauwhede1989centre}, to the
  center-unstable case. Parts of the results are also stated as
  Theorem 2.9 and 3.22 of Section 2 of \cite{haragus2011local}.

  The key point is to note that by the Duhamel formulation a solution
  with the growth bound is a fixed point of the map $T$ given by
  \begin{equation*}
    (T \tilde{u})(t)
    = e^{t \tilde{L}} \tilde{P}_{cu} \tilde{u}_0
    - \int_{t}^{0} e^{(t-s) \tilde{L}} \tilde{P}_{cu}
    \tilde{R}_{\delta}(\tilde{u}(s)) \dd s
    + \int_{-\infty}^{t} e^{(t-s) \tilde{L}} \tilde{P}_{s}
    \tilde{R}_{\delta}(\tilde{u}(s)) \dd s
  \end{equation*}
  for $t \in \R^-$ on the space $C(\R^-\times \N \times \R \times \R)$
  with norm
  $\| \tilde{u} \| = \sup_{t \in \R^-} e^{\mu t}
  \|\tilde{u}(t,\cdot)\|_{\tilde{\mathcal{Z}}^a}$.
  Here note that on $\tilde{\mathcal{Z}}^a_{cu}$ the linear evolution
  $e^{t\tilde{L}}$ restricts to a finite dimensional matrix
  exponential, which has a well-defined inverse $e^{-t \tilde{L}}$.

  For small enough $\delta$ the map $T$ is a contraction with a unique
  fixed point depending on $\tilde{P}_{cu}\tilde{u}_0$ so that there
  exists
  $\tilde{\psi} : \tilde{\mathcal{Z}}^a_{cu} \mapsto
  \tilde{\mathcal{Z}}^a_{s}$
  such that the unique solution is
  $\tilde{P}_{cu} \tilde{u}_0 +
  \tilde{\psi}(\tilde{P}_{cu}\tilde{u}_0)$
  \cite[Theorem 1]{vanderbauwhede1992center}. Hence the set
  $\mathcal{M}$ is a reduced manifold with the given form. Moreover,
  the growth condition is invariant under the evolution.

  By a Fibre contraction argument (\cite[Theorem
  2]{vanderbauwhede1992center} and \cite[Section
  1.3]{vanderbauwhede1989centre}), for small enough $\delta$, the map
  $\tilde{\psi}$ has the claimed regularity and the derivatives can be
  computed explicitly at $0$ with the given form.

  As discussed in \cite[Section 1.5]{vanderbauwhede1989centre} and
  \cite[Theorem 3.22 of Chapter 2]{haragus2011local} a similar fixed
  point argument shows the exponential attractiveness.
\end{proof}

Hence, with a sufficient localisation, we can prove
\Cref{thm:center-unstable-kuramoto}.
\begin{proof}[Proof of \Cref{thm:center-unstable-kuramoto}]
  Apply \Cref{thm:center-unstable-localised} with a small enough
  $\delta$. Then for $|\epsilon| \le \delta$ and $\| u \| \le \delta$
  the nonlinearity of $\tilde{R}_{\delta}$ agrees with the
  nonlinearity of the Kuramoto equation, i.e. a solution to the
  localised evolution is also a solution to the original problem in
  the restricted region.

  The map $\tilde{\psi}$ maps into $\tilde{\mathcal{Z}}^a_s$ and thus
  has the form $\tilde{\psi}(u,\epsilon) = (\psi(u,\epsilon),0)$ for
  $\psi$ mapping into $\mathcal{Z}^a_s$. As $\epsilon$ is constant,
  the reduced manifold must have the claimed form and the result
  follows from the the uniqueness (\Cref{thm:fourier-uniqueness}).
\end{proof}

\begin{remark}
  The Kuramoto equation has the rotation symmetry, i.e. if $u$ is a
  solution to the Kuramoto equation, then
  $\bar{u}(t,l,\xi) = e^{il\alpha} u(t,l,\xi)$ is again a
  solution. This symmetry is also satisfied by the localised
  nonlinearity and thus by the reduced manifold.
\end{remark}

\subsection{Bifurcation analysis}

The center-manifold reduction can be used to determine the bifurcation
behaviour by studying the evolution on the reduced manifold. The
resulting behaviour crucially depends on the velocity distribution and
\cite{crawford1994amplitude} contains several examples based on the
center-manifold reduction with noise which at this point is very
similar. The Penrose diagrams
(cf. \Cref{fig:gaussianStability,fig:gaussianTwoStability}) already
show the dimension of the reduced manifold as the covering number,
e.g. this shows for the example of \Cref{fig:gaussianTwoStability}
that at the critical coupling two eigenmodes appear.

As an example we repeat Chiba's analysis \cite{chiba2013ergodic} for
the bifurcation of the Gaussian distribution. In this example the
density is given by
\begin{equation*}
  g(\omega) = \frac{1}{\sqrt{2\pi}}
  e^{-\omega^2/2},
\end{equation*}
and the Fourier transform is
\begin{equation*}
  \hat{g}(\xi) = e^{-\xi^2/2}.
\end{equation*}
From the discussion in \Cref{sec:linearised-stability}, we see that
the critical coupling is
\begin{equation*}
  K_c = \frac{2}{\pi g(0)} =  \frac{4}{\sqrt{2\pi}}
\end{equation*}
and that at the critical coupling $K_c$ there exists a single
eigenvalue at $\lambda = 0$ for small enough $a$.

Recall, that the eigenvector $z_{0,0}$ is given by
\begin{equation*}
  z_{0,0}(1,\xi) = \frac{K_c}{2} \int_{\xi}^{\infty} \hat{g}(\zeta)
  \dd \zeta.
\end{equation*}
Hence $z_{0,0}(1,0) = 1$ and
\begin{align*}
  \alpha_{0,0}(z_{0,0})
  &= \frac{K_c}{2} \int_0^\infty \int_{\xi}^{\infty} \hat{g}(\zeta) \dd \zeta
    \dd \xi \\
  &= \frac{K_c}{2},
\end{align*}
so that
\begin{equation*}
  P_{cu}(u) = \frac{2}{K_c} \alpha_{0,0}(u) z_{0,0}.
\end{equation*}

Now apply \Cref{thm:center-unstable-kuramoto} with sufficiently small
$a > 0$ such that $\lambda_1=0$ is the only eigenmode and with $k \ge
3$. Then for $\tilde{u} = (u,\epsilon) \in \tilde{\mathcal{Z}}^a_{cu}$
\begin{equation*}
  \psi(\tilde{u}) = \frac{1}{2} b(\tilde{u}) + O(\| \tilde{u} \|^3)
\end{equation*}
where
\begin{equation*}
  b(\tilde{u}) = D^2\psi|_0(\tilde{u},\tilde{u})
  = \int_{-\infty}^{0} e^{-t\tilde{L}} \tilde{P}_s D^2
  \tilde{R}|_{0}(\tilde{u},\tilde{u}) \dd t.
\end{equation*}
For this we find that only the first two spatial modes are
non-vanishing with
\begin{equation*}
  b(\tilde{u})(1,\xi) = O(|\epsilon| \|u \|)
\end{equation*}
and
\begin{equation*}
  b(\tilde{u})(2,\xi) = \int_{-\infty}^{0} 2 K_c u(1,\xi-2t) u(1,0)
  \dd t.
\end{equation*}
Hence we can compute
\begin{equation*}
  N(u+\psi(u,\epsilon))(1,\xi)
  = \frac{\epsilon}{2} u(1,0) \hat{g}(\xi)
  - \frac{K_c}{2} \conj{u(1,0)} \frac{b(\tilde{u})(2,\xi)}{2}
  + O(|\epsilon|^2 \| u \|, \| \tilde{u} \|^4).
\end{equation*}
As $\tilde{u} = u(1,0) z_{0,0}$ we find
\begin{align*}
  \int_0^\infty b(\tilde{u})(2,\xi) \dd \xi
  &= (u(1,0))^2 \int_0^\infty \int_{-\infty}^{0} 2K_c
    z_{0,0}(1,\xi-2t) z_{0,0}(1,0) \dd t \dd \xi \\
  &= K_c^2 (u(1,0))^2 \int_{\xi=0}^\infty \int_{t=0}^\infty
    \int_{\zeta=\xi+2t}^\infty \hat{g}(\zeta) \dd \zeta \dd t \dd \xi
  \\
  &= K_c^2 (u(1,0))^2 \int_{\zeta=0}^\infty \frac{\zeta^2}{4}
    \hat{g}(\zeta) \dd \zeta \\
  &= \frac{\sqrt{2\pi}}{8}.
\end{align*}
Hence we find
\begin{equation*}
  \alpha_{0,0}(\tilde{R}(u+\psi(u,\epsilon)))
  = \frac{\sqrt{2\pi}}{4} \epsilon u(1,0)
  - \frac{1}{\pi} |u(1,0)|^2 u(1,0)
  + O(|\epsilon|^2 \| u \|, \| \tilde{u} \|^4).
\end{equation*}
On the reduced manifold the solution is $\beta(t) z_{0,0} +
\psi(\beta(t) z_{0,0},\epsilon)$, which evolves by the previous
computation as
\begin{equation*}
  \frac{K_c}{2} \partial_t \beta
  = \frac{\sqrt{2\pi}}{4} \epsilon \beta
  - \frac{1}{\pi} |\beta|^2 \beta
  + O(|\epsilon|^2 \| u \|, \| \tilde{u} \|^4).
\end{equation*}
For small enough $\epsilon \le 0$, thus $\beta = 0$ is a stable fixed
point, i.e. the incoherent state is stable. For small enough
$\epsilon > 0$, the zero solution is unstable, but the set $|\beta| =
\beta_c$ are stable attractors where
\begin{equation*}
  \beta_c = \sqrt{\frac{\pi}{4}\sqrt{2\pi}} \sqrt{\epsilon} + O(\epsilon).
\end{equation*}
Hence we have proved the claimed bifurcation from the incoherent state
and that for small enough $\epsilon$ the appearing stationary states
are nonlinearly stable. Since $z_{0,0}(1,0)=1$, the value of $\beta$
agrees with the order parameter so that we have found the known
result.

\section{Boundedness and convergence in the exponential norms}
\label{sec:norms}

The exponential norm $L^\infty(\R,\exp_a)$ can be related to an
analytic extension in a strip $\{z : 0 \le \Im z < a\}$.
\begin{lemma}
  Let $f \in L^1(\R) \cap C(\R)$ with Fourier transform
  $\hat{f} \in L^1(\R)$ satisfying $\hat{f} \in
  L^\infty(\R,\exp_a)$.
  Then $f$ has an analytic continuation to $\{z : 0 \le \Im z < a\}$.
\end{lemma}
\begin{proof}
  For $z \in \R$ the Fourier inversion formula shows
  \begin{equation*}
    f(z) = \frac{1}{2\pi} \int_{\R} e^{-iz\xi} \hat{f}(\xi) \dd \xi.
  \end{equation*}
  By the assumed supremum bound on $\hat{f}$ for $\xi \ge 0$ and the
  bound $|\hat{f}(\xi)| \le \| f \|_1$ for $\xi \le 0$, this
  definition extends to the given range and has a complex derivative
  given by differentiating under the integral sign.
\end{proof}

\begin{lemma}
  \label{thm:bound-fourier-continuation}
  Let $f \in L^1(\R)$ and suppose $f$ has an analytic continuation to
  $\{z: 0 \le \Im z \le a\}$ such that $|f(z)| \to 0$ uniformly as
  $|\Re z| \to \infty$ and $f(\cdot+ia) \in L^1(\R)$. Then the Fourier
  transform $\hat{f}$ is bounded as
  \begin{equation*}
    \| \hat{f} \|_{L^\infty(\R,\exp_a)}
    \le \| f(\cdot + ia) \|_{L^1(\R)}.
  \end{equation*}
\end{lemma}
\begin{proof}
  By the assumed decay we can deform the contour of integration as
  \begin{equation*}
    \hat{f}(\xi) = \int_{\R} e^{i\xi y} f(y) \dd y
    = \int_{\R} e^{i\xi (y+ia)} f(y+ia) \dd y
    = e^{-a\xi} \int_{\R} e^{i\xi y} f(y+ia) \dd y.
  \end{equation*}
  Bounding the integral with the $L^1$ norm of $f$ shows the claimed
  result.
\end{proof}

We can use this to show the finite $\mathcal{Z}^a$ norm of partially
locked states (see \cite[Chapter 4]{strogatz2000from} for a brief
review of these states).
\begin{thm}
  Suppose a velocity marginal with density $g \in L^1(\R)$ and that
  $g$ has an analytic continuation to $\{z: 0 \le \Im z \le a\}$ such
  that $|g(z)| \to 0$ uniformly as $|\Re z| \to \infty$ and
  $g(\cdot+ia) \in L^1(\R)$. Then a partially locked state
  $\rho_{locked}$ has Fourier transform $u \in \mathcal{Z}^a$. If
  furthermore $|\eta| \le a / (\sqrt{2}K)$, then
  \begin{equation*}
    \| u \|_{\mathcal{Z}^a}
    \le \frac{K |\eta|}{2} \| g(\cdot + ia) \|_{L^1(\R)}.
  \end{equation*}
\end{thm}
The assumption on the velocity distribution is for example satisfied
for a Gaussian distribution. This estimate shows that partially locked
states with small order parameter $\eta$ are small perturbations in
$\mathcal{Z}^a$.
\begin{proof}
  Suppose a partially locked state with order parameter $\eta$. Then
  for $|\omega| < K |\eta|$ all oscillators are trapped at $\theta_+$
  satisfying
  \begin{equation*}
    e^{i\theta_+}
    = \frac{\sqrt{K^2 |\eta|^2 - \omega^2}}{K \conj{\eta}}
    + \frac{i \omega}{K \conj{\eta}}.
  \end{equation*}
  For $|\omega| > K |\eta|$ the oscillators are distributed
  proportionally to the density
  \begin{equation*}
    \frac{1}{\omega + \frac{K}{2i} \left(\eta e^{-i\theta}
        - \conj{\eta} e^{i\theta}\right)}.
  \end{equation*}
  Consider the Fourier transform $\hat{f}$ with respect to the phase
  angle
  \begin{equation*}
    \hat{f}(l,\omega) = \int_{\theta \in \T} e^{il\theta}
    \rho_{locked}(\dd \theta,\omega).
  \end{equation*}
  By Cauchy's residue theorem
  \begin{equation*}
    \hat{f}(l,\omega) = g(\omega)
    \frac{\eta}{|\eta|}
    \left( \beta\left(\frac{\omega}{K|\eta|}\right) \right)^l
  \end{equation*}
  where $\beta$ is given by
  \begin{equation*}
    \beta(z) =
    \begin{cases}
      i z \left(1 - \sqrt{1-\frac{1}{z^2}}\right) & \text{ for $|z|
        \ge 1$}, \\
      i z + \sqrt{1-z^2} & \text{ for $|z| < 1$}.
    \end{cases}
  \end{equation*}
  Now $\beta$ has an analytic continuation to the complex upper half
  plane and is bounded there by $1$. Hence
  \Cref{thm:bound-fourier-continuation} shows for all $l \in \N$
  \begin{equation*}
    \| u(l,\cdot) \|_{L^\infty(\R,\exp_a)}
    \le \| g(\cdot + ia) \|_{L^1(\R)},
  \end{equation*}
  which implies that $u \in \mathcal{Z}^a$.

  For $|z| \ge \sqrt{2}$ we can estimate by the series expansion
  \begin{equation*}
    |\beta(z)| \le \frac{1}{2|z|} \sum_{i=0}^{\infty} |z|^{-2i}
    \le \frac{1}{|z|}.
  \end{equation*}
  Hence we can estimate for $|\eta| \le a/(\sqrt{2} K)$ that
  \begin{equation*}
    \| u(l,\cdot) \|_{L^\infty(\R,\exp_a)}
    \le
    \sup_{\omega\in\R} \left|\beta\left(\frac{\omega + ia}{K|\eta|}\right) \right|^l
    \| g(\cdot + ia) \|_{L^1(\R)}
    \le \frac{K|\eta|}{a}
    \| g(\cdot + ia) \|_{L^1(\R)},
  \end{equation*}
  which is the claimed bound.
\end{proof}

We can relate convergence in $\mathcal{Z}^a$ as found in the
center-unstable manifold reduction and the exponential stability with
weak convergence.
\begin{thm}
  Let $\rho_{\cdot} \in C_{\mathcal{M}}$ be a solution to the Kuramoto
  equation with initial data $\rho_{in} \in \mathcal{M}(\Gamma)$ and
  Fourier transform $u$. If $u(t,\cdot,\cdot) \to v$ in
  $\mathcal{Z}^a$ as $t \to \infty$ where $v$ is the Fourier transform
  of $\rho_{\infty} \in \mathcal{M}(\Gamma)$, then for every $\phi \in
  H^4(\Gamma)$ the integral $\rho_t(\phi)$ converges to
  $\rho_{\infty}(\phi)$ as $t \to \infty$.
\end{thm}
\begin{proof}
  Since the velocity marginal is conserved, wlog assume $\int_{\theta
    \in \T} \phi(\theta,\omega) \dd \theta = 0$ for all $\omega \in
  \R$ and $\phi$ to be real-valued. Then by Plancherel theorem
  \begin{equation*}
    \rho_t(\phi) = 2 \Re \left( \sum_{l\ge1} \int_{\xi\in\R} u(t,l,\xi)
      \conj{\hat{\phi}(l,\xi)} \dd \xi \right),
  \end{equation*}
  where $\hat{\phi}$ is the Fourier transform of $\phi$. Since $\phi
  \in H^4$, we have $\hat{\phi} \in L^1(\N \times \R)$. Hence for
  every $\epsilon > 0$, there exists some $M$ such that
  \begin{equation*}
    \sum_{l\ge 1} \int_{-\infty}^{M} |\hat{\phi}(l,\xi)| \dd \xi
    \le \frac{\epsilon}{8}.
  \end{equation*}
  Since $u$ and $v$ correspond to probability measures, they are
  bounded by $1$ so that
  \begin{equation*}
    \left|
      \sum_{l\ge1} \int_{-\infty}^{M} (u(t,l,\xi)-v(l,\xi))
      \conj{\hat{\phi}(l,\xi)} \dd \xi
    \right| \le \frac{\epsilon}{4}.
  \end{equation*}
  For $\xi \ge M$ we have a control by $\mathcal{Z}^a$
  \begin{equation*}
    \left|
      \sum_{l\ge1} \int_{M}^{\infty} (u(t,l,\xi)-v(l,\xi))
      \conj{\hat{\phi}(l,\xi)} \dd \xi
    \right|
    \le
    e^{-aM} \|u - v\|_{\mathcal{Z}^a} \| \hat{\phi} \|_1.
  \end{equation*}
  Hence for large enough $t$ we have
  $|\rho_t(\phi) - \rho_{\infty}(\phi)| \le \epsilon$. Since
  $\epsilon$ is arbitrary, this shows the claimed convergence.
\end{proof}

In the stability theorem \ref{thm:algebraic-stability} we only control
a norm in half of the Fourier coefficients which also looses control
for higher spatial modes. However, this controls the order parameter
$\eta$ and in particular it shows $\int_0^\infty |\eta(t)| \dd t <
\infty$. With this knowledge we can go back to the original equation
to deduce convergence properties. As an example we show a convergence
in the gliding frame.

In the gliding frame the position of each oscillator is corrected by
the effect of the free transport, i.e. if $\rho_{\cdot}$ is the
solution of the Kuramoto equation which has a density $f$, then the
density $h$ in the gliding frame is given by
\begin{equation*}
  h(t,\theta,\omega) = f(t,\theta+t\omega,\omega).
\end{equation*}
If $f \in C^1(\R^+ \times \Gamma)$, then $f$ satisfies classically
\begin{equation*}
  \partial_t f(t,\theta,\omega) + \partial_{\theta}
  \left[ \omega + \left(\frac{K}{2i}
      \left(\eta(t)\, e^{-i\theta} - \conj{\eta(t)}\, e^{i\theta}\right)\right)
    f(t,\theta) \right] = 0
\end{equation*}
so that then $h \in C^1(\R^+ \times \Gamma)$ and $h$ satisfied
\begin{equation}
  \label{eq:gliding-frame-pde}
  \partial_t h(t,\theta,\omega) + \partial_\theta
  \left[ \frac{K}{2i}
    \left(\eta(t)\, e^{-i(\theta+t\omega)} - \conj{\eta(t)}\, e^{i(\theta+t\omega)}\right)
    h(t,\theta,\omega) \right] = 0.
\end{equation}
\begin{thm}
  \label{thm:gliding-control}
  Suppose $\rho_{in} \in \mathcal{M}(\Gamma)$ has density
  $f_{in} \in C^2(\Gamma)$ and let $\rho_{\cdot}$ be the solution of
  the Kuramoto equation with initial data $\rho_{in}$. Then
  $\rho_{\cdot}$ has a gliding frame density $h \in C^2(\R^+ \times
  \Gamma)$ and for every $t \in \R^+$ and $\omega \in \R$ holds
  \begin{equation*}
    \int_{\T} (h(t,\theta,\omega))^2 \dd \theta
    \le e^{K \int_0^t |\eta(s)| \dd s}
    \int_{\T} (f_{in}(\theta,\omega))^2 \dd \theta
  \end{equation*}
  and
  \begin{align*}
    &\int_{\T} (\partial_{\theta}h(t,\theta,\omega))^2 \dd \theta \\
    &\le e^{3K \int_0^T |\eta(s)| \dd s}
    \left[
      \int_{\T} (\partial_{\theta} f_{in}(\theta,\omega))^2 \dd \theta
      + K \int_0^t |\eta(s)| \dd s
      \sup_{s \in [0,t]} \int_{\T} (h(t,\theta,\omega))^2 \dd \theta
    \right].
  \end{align*}
  If $\int_0^{\infty} |\eta(t)| \dd t < \infty$, then for every fixed
  $\omega \in \R$ the function $h(t,\cdot,\omega)$ converges in
  $L^2(\T)$ as $t \to \infty$.

  If additionally $f_{in} \in L^2(\Gamma)$ and $\partial_{\theta}
  f_{in} \in L^2(\Gamma)$, then $h(t,\cdot,\cdot)$ converges in
  $L^2(\Gamma)$ as $t \to \infty$.
\end{thm}
\begin{proof}
  Along the evolution the regularity is propagated
  \cite{lancellotti2005vlasov} so that $\rho_{\cdot}$ has a density $f
  \in C^2(\R^+ \times \Gamma)$. Hence $h \in C^2(\R^+ \times \Gamma)$
  and $h$ satisfies \Cref{eq:gliding-frame-pde}.

  For the bounds, fix $\omega$ and let
  \begin{equation*}
    z(t) = \int_{\theta \in \T} (h(t,\theta,\omega))^2 \dd \theta.
  \end{equation*}
  By the assumed regularity, we can differentiate under the integral
  sign to find
  \begin{equation*}
    \frac{\dd}{\dd t} z(t)
    = \int_{\theta \in \T} \frac{K}{2}
    \left(\eta(t)\, e^{-i(\theta+t\omega)} + \conj{\eta(t)}\,
      e^{i(\theta+t\omega)}\right)
    (h(t,\theta,\omega))^2 \dd \theta
    \le K |\eta(t)| z(t).
  \end{equation*}
  Hence Gronwall's inequality shows as claimed
  \begin{equation*}
    z(t) \le e^{K \int_0^t |\eta(s)| \dd s} z(0)
    = e^{K \int_0^t |\eta(s)| \dd s}
    \int_{\T} (f_{in}(\theta,\omega))^2 \dd \theta.
  \end{equation*}
  Similar, we consider
  \begin{equation*}
    y(t) = \int_{\theta \in \T} (\partial_{\theta}h(t,\theta,\omega))^2 \dd \theta
  \end{equation*}
  and find
  \begin{align*}
    \frac{\dd}{\dd t} y(t)
    &= \frac{K}{2}
      \int_{\theta \in \T}
      \left[(h(t,\theta,\omega))^2
      + 3 (\partial_{\theta}h(t,\theta,\omega))^2\right]
      \left(\eta(t)\, e^{-i(\theta+t\omega)} + \conj{\eta(t)}\,
      e^{i(\theta+t\omega)}\right) \dd \theta \\
    &\le K |\eta(t)| z(t) + 3K |\eta(t)| y(t).
  \end{align*}
  Then Gronwall's inequality shows the claimed bound.

  If $\int_0^t |\eta(t)| \dd t < \infty$, these bounds show that there
  exist constants $C_1,C_2 \in \R$ such that for every $\omega \in \R$
  \begin{equation*}
    \int_{\theta \in \T} (h(t,\theta,\omega))^2 \dd \theta
    \le C_1 \int_{\theta \in \T} (f_{in}(\theta,\omega))^2 \dd \theta
  \end{equation*}
  and
  \begin{equation*}
    \int_{\theta \in \T} (\partial_{\theta}h(t,\theta,\omega))^2 \dd \theta
    \le C_2 \int_{\theta \in \T}
    \left[(f_{in}(\theta,\omega))^2 +
      (\partial_{\theta}f_{in}(\theta,\omega))^2\right] \dd \theta.
  \end{equation*}
  In particular the bounds are finite as the integral is over a
  compact domain and we assume $f_{in} \in C^2(\Gamma)$.

  Now consider for $t \ge s \ge 0$
  \begin{equation*}
    D(t,s) = \int_{\theta \in \T}
    (h(t,\theta,\omega)-h(s,\theta,\omega))^2 \dd \theta.
  \end{equation*}
  By the assumed regularity we can differentiate under the integral
  sign and find
  \begin{align*}
    &\partial_t D(t,s) \\
    &= \int_{\theta \in \T} 2 [\partial_{\theta}
    h(t,\theta,\omega)-\partial_{\theta} h(s,\theta,\omega)]
    \frac{K}{2i} \left(\eta(t)\, e^{-i(\theta+t\omega)} -
      \conj{\eta(t)}\, e^{i(\theta+t\omega)}\right)
    h(t,\theta,\omega) \dd \theta \\
    &\le 2 K |\eta(t)|
      (\sqrt{y(t)} + \sqrt{y(s)})
      \sqrt{z(t)}.
  \end{align*}
  Hence
  \begin{equation*}
    D(t,s) \le C \int_s^t |\eta(\bar{s})| \dd \bar{s}
  \end{equation*}
  for a constant $C$. As $\int_0^\infty |\eta(t)| \dd t < \infty$,
  this shows the claimed convergence in $L^2(\T)$ as $t \to \infty$.

  If $f_{in} \in L^2(\Gamma)$ and
  $\partial_{\theta} f_{in} \in L^2(\Gamma)$, we can integrate the
  previous inequality over $\omega$ to show the convergence of
  $h(t,\cdot,\cdot)$ in $L^2(\Gamma)$ as $t \to \infty$.
\end{proof}

The convergence in the gliding frame implies for example weak
convergence in the normal setup.

\begin{thm}
  Suppose the assumption of \Cref{thm:gliding-control} with
  $\int_0^\infty |\eta(t)| \dd t < \infty$ and
  $f_{in} \in L^2(\Gamma)$ and
  $\partial_{\theta} f_{in} \in L^2(\Gamma)$. Then for every
  $\phi \in L^2(\Gamma)$
  \begin{equation*}
    \int_{\R} \int_{\T} f(t,\theta,\omega) \phi(\theta,\omega) \dd
    \theta \dd \omega
    \to \int_{\R} \int_{\T} \frac{g(\omega)}{2\pi} \phi(\theta,\omega) \dd
    \theta \dd \omega.
  \end{equation*}
\end{thm}
\begin{proof}
  By the previous theorem, the gliding frame density converges in
  $L^2(\Gamma)$. Let $h_{\infty}$ be the limit, i.e.
  \begin{equation*}
    h(t,\cdot,\cdot) \to h_{\infty}(\cdot,\cdot)
  \end{equation*}
  holds in $L^2(\Gamma)$ as $t \to \infty$.

  The inner product can be expressed in terms of $h$ as
  \begin{equation*}
    \int_{\R} \int_{\T} f(t,\theta,\omega) \phi(\theta,\omega) \dd
    \theta \dd \omega
    =
    \int_{\R} \int_{\T} h(t,\theta,\omega) \phi(\theta+t\omega,\omega) \dd
    \theta \dd \omega.
  \end{equation*}
  Given $\epsilon > 0$, we then have for large enough $t$ by the
  convergence in $L^2$
  \begin{equation*}
    \left|\int_{\R} \int_{\T} f(t,\theta,\omega) \phi(\theta,\omega) \dd
    \theta \dd \omega
    -
    \int_{\R} \int_{\T} h_{\infty}(\theta,\omega) \phi(\theta+t\omega,\omega) \dd
    \theta \dd \omega
    \right| \le \frac{\epsilon}{2}.
  \end{equation*}

  Let $\hat{h}_{\infty}$ and $\hat{\phi}$ be the Fourier transform in
  both variables $\theta$ and $\omega$. Then the Plancherel theorem shows
  \begin{equation*}
    \int_{\R} \int_{\T} h_{\infty}(\theta,\omega) \phi(\theta+t\omega,\omega) \dd
    \theta \dd \omega
    = \sum_{k\in \Z} \int_{\xi\in\R}
    \hat{h}_{\infty}(k,\xi)
    \hat{\phi}(k,\xi-kt) \dd \xi.
  \end{equation*}
  The velocity marginal $g$ is constant so that
  $\hat{h}_{\infty}(0,\xi) = \hat{g}(\xi)$ which shows
  \begin{equation*}
    \int_{\xi\in\R}
    \hat{h}_{\infty}(0,\xi)
    \hat{\phi}(0,\xi) \dd \xi
    = \int_{\omega\in\R} \int_{\theta\in\T} \frac{g(\omega)}{2\pi}
    \phi(\theta,\omega) \dd \theta \dd \omega.
  \end{equation*}
  For $k \not = 0$ we find that
  \begin{equation*}
    \left|\int_{\xi\in\R}
      \hat{h}_{\infty}(k,\xi)
      \hat{\phi}(k,\xi-kt) \dd \xi\right|
    \le \| \hat{h}_{\infty}(k,\cdot) \|_2
    \| \hat{\phi}(k,\cdot) \|_2
  \end{equation*}
  and
  \begin{equation*}
    \int_{\xi\in\R}
    \hat{h}_{\infty}(k,\xi)
    \hat{\phi}(k,\xi-kt) \dd \xi
    \to 0
  \end{equation*}
  as $t \to \infty$.

  Since $\hat{h} \in L^2$ and $\hat{\phi} \in L^2$, we have the bound
  \begin{equation*}
    \sum_{k\in\Z} \| \hat{h}_{\infty}(k,\cdot) \|_2
    \| \hat{\phi}(k,\cdot) \|_2
    < \infty,
  \end{equation*}
  so that dominated convergence shows as $t \to \infty$
  \begin{equation*}
    \sum_{k\in \Z} \int_{\xi\in\R}
    \hat{h}_{\infty}(k,\xi)
    \hat{\phi}(k,\xi-kt) \dd \xi
    \to
    \int_{\xi\in\R}
    \hat{h}_{\infty}(0,\xi)
    \hat{\phi}(0,\xi) \dd \xi,
  \end{equation*}
  which is the claimed limit.
\end{proof}

\bibliographystyle{plain}
\bibliography{lit}

\begin{thebibliography}{10}

\bibitem{acebron1998breaking}
J.~A. {Acebr{\'o}n}, L.~L. {Bonilla}, S.~{de Leo}, and R.~{Spigler}.
\newblock {Breaking the symmetry in bimodal frequency distributions of globally
  coupled oscillators}.
\newblock {\em Phys.~Rev.~E}, 57:5287--5290, May 1998.

\bibitem{acebron2005synchronization}
J.~A. {Acebr{\'o}n}, L.~L. {Bonilla}, C.~J. {P{\'e}rez Vicente}, F.~{Ritort},
  and R.~{Spigler}.
\newblock {The Kuramoto model: A simple paradigm for synchronization
  phenomena}.
\newblock {\em Reviews of Modern Physics}, 77:137--185, January 2005.

\bibitem{bedrossian2013inviscid}
J.~{Bedrossian} and N.~{Masmoudi}.
\newblock {Inviscid damping and the asymptotic stability of planar shear flows
  in the 2D Euler equations}.
\newblock {\em ArXiv e-prints}, June 2013.

\bibitem{bedrossian2013landau}
J.~{Bedrossian}, N.~{Masmoudi}, and C.~{Mouhot}.
\newblock {Landau damping: paraproducts and Gevrey regularity}.
\newblock {\em ArXiv e-prints}, November 2013.

\bibitem{benedetto2014exponential}
D.~{Benedetto}, E.~{Caglioti}, and U.~{Montemagno}.
\newblock {Exponential dephasing of oscillators in the Kinetic Kuramoto Model}.
\newblock {\em ArXiv e-prints}, December 2014.

\bibitem{bonilla1998time}
L.~L. Bonilla, C.~J. P{\'e}rez~Vicente, and R.~Spigler.
\newblock Time-periodic phases in populations of nonlinearly coupled
  oscillators with bimodal frequency distributions.
\newblock {\em Phys. D}, 113(1):79--97, 1998.

\bibitem{bonilla1992nonlinear}
Luis~L. Bonilla, John~C. Neu, and Renato Spigler.
\newblock Nonlinear stability of incoherence and collective synchronization in
  a population of coupled oscillators.
\newblock {\em J. Statist. Phys.}, 67(1-2):313--330, 1992.

\bibitem{braun1977vlasov}
W.~Braun and K.~Hepp.
\newblock The {V}lasov dynamics and its fluctuations in the {$1/N$} limit of
  interacting classical particles.
\newblock {\em Comm. Math. Phys.}, 56(2):101--113, 1977.

\bibitem{carrillo2014contractivity}
Jos{\'e}~A. Carrillo, Young-Pil Choi, Seung-Yeal Ha, Moon-Jin Kang, and
  Yongduck Kim.
\newblock Contractivity of transport distances for the kinetic {K}uramoto
  equation.
\newblock {\em J. Stat. Phys.}, 156(2):395--415, 2014.

\bibitem{chiba2013ergodic}
Hayato Chiba.
\newblock {A proof of the Kuramoto conjecture for a bifurcation structure of
  the infinite-dimensional Kuramoto model}.
\newblock {\em Ergodic Theory and Dynamical Systems}, FirstView:1–73, 10
  2013.

\bibitem{chiba2013continuous}
Hayato Chiba.
\newblock Continuous limit and the moments system for the globally coupled
  phase oscillators.
\newblock {\em Discrete Contin. Dyn. Syst.}, 33(5):1891--1903, 2013.

\bibitem{chiba2011center}
Hayato Chiba and Isao Nishikawa.
\newblock Center manifold reduction for large populations of globally coupled
  phase oscillators.
\newblock {\em Chaos}, 21(4):043103, 10, 2011.

\bibitem{crawford1995scaling}
John Crawford.
\newblock Scaling and singularities in the entrainment of globally coupled
  oscillators.
\newblock {\em Phys. Rev. Lett.}, 74:4341--4344, May 1995.

\bibitem{crawford1999synchronization}
John~D. Crawford and K.~T.~R. Davies.
\newblock Synchronization of globally coupled phase oscillators: singularities
  and scaling for general couplings.
\newblock {\em Phys. D}, 125(1-2):1--46, 1999.

\bibitem{crawford1994amplitude}
John~David Crawford.
\newblock Amplitude expansions for instabilities in populations of
  globally-coupled oscillators.
\newblock {\em J. Statist. Phys.}, 74(5-6):1047--1084, 1994.

\bibitem{dobrushin1979vlasov}
R.~L. Dobru{\v{s}}in.
\newblock Vlasov equations.
\newblock {\em Funktsional. Anal. i Prilozhen.}, 13(2):48--58, 96, 1979.

\bibitem{faou2014landau}
E.~{Faou} and F.~{Rousset}.
\newblock {Landau damping in Sobolev spaces for the Vlasov-HMF model}.
\newblock {\em ArXiv e-prints}, March 2014.

\bibitem{fernandez2014landau}
B.~{Fernandez}, D.~{G{\'e}rard-Varet}, and G.~{Giacomin}.
\newblock {Landau damping in the Kuramoto model}.
\newblock {\em ArXiv e-prints}, October 2014.

\bibitem{gripenberg1990volterra}
G.~Gripenberg, S.-O. Londen, and O.~Staffans.
\newblock {\em Volterra integral and functional equations}, volume~34 of {\em
  Encyclopedia of Mathematics and its Applications}.
\newblock Cambridge University Press, Cambridge, 1990.

\bibitem{ha2010complete}
Seung-Yeal Ha, Taeyoung Ha, and Jong-Ho Kim.
\newblock On the complete synchronization of the {K}uramoto phase model.
\newblock {\em Phys. D}, 239(17):1692--1700, 2010.

\bibitem{haragus2011local}
Mariana Haragus and G{\'e}rard Iooss.
\newblock {\em Local bifurcations, center manifolds, and normal forms in
  infinite-dimensional dynamical systems}.
\newblock Universitext. Springer-Verlag London, Ltd., London; EDP Sciences, Les
  Ulis, 2011.

\bibitem{kuramoto1984chemical}
Y.~Kuramoto.
\newblock {\em Chemical oscillations, waves, and turbulence}, volume~19 of {\em
  Springer Series in Synergetics}.
\newblock Springer-Verlag, Berlin, 1984.

\bibitem{kuramoto1975self}
Yoshiki Kuramoto.
\newblock Self-entrainment of a population of coupled non-linear oscillators.
\newblock In {\em International {S}ymposium on {M}athematical {P}roblems in
  {T}heoretical {P}hysics ({K}yoto {U}niv., {K}yoto, 1975)}, pages 420--422.
  Lecture Notes in Phys., 39. Springer, Berlin, 1975.

\bibitem{lancellotti2005vlasov}
Carlo Lancellotti.
\newblock On the {V}lasov limit for systems of nonlinearly coupled oscillators
  without noise.
\newblock {\em Transport Theory Statist. Phys.}, 34(7):523--535, 2005.

\bibitem{martens2009exact}
E.~A. Martens, E.~Barreto, S.~H. Strogatz, E.~Ott, P.~So, and T.~M. Antonsen.
\newblock Exact results for the {K}uramoto model with a bimodal frequency
  distribution.
\newblock {\em Phys. Rev. E (3)}, 79(2):026204, 11, 2009.

\bibitem{mouhot2011landau}
Cl{\'e}ment Mouhot and C{\'e}dric Villani.
\newblock On {L}andau damping.
\newblock {\em Acta Math.}, 207(1):29--201, 2011.

\bibitem{neunzert1984introduction}
H.~Neunzert.
\newblock An introduction to the nonlinear {B}oltzmann-{V}lasov equation.
\newblock In {\em Kinetic theories and the {B}oltzmann equation ({M}ontecatini,
  1981)}, volume 1048 of {\em Lecture Notes in Math.}, pages 60--110. Springer,
  Berlin, 1984.

\bibitem{ott2008low}
Edward Ott and Thomas~M. Antonsen.
\newblock Low dimensional behavior of large systems of globally coupled
  oscillators.
\newblock {\em Chaos}, 18(3):037113, 6, 2008.

\bibitem{ott2009long}
Edward Ott and Thomas~M. Antonsen.
\newblock Long time evolution of phase oscillator systems.
\newblock {\em Chaos}, 19(2):023117, 5, 2009.

\bibitem{ott2011comment}
Edward Ott, Brian~R. Hunt, and Thomas~M. Antonsen, Jr.
\newblock Comment on ``{L}ong time evolution of phase oscillator systems''
  [{C}haos 19, 023117 (2009)] [mr2548751].
\newblock {\em Chaos}, 21(2):025112, 2, 2011.

\bibitem{penrose1960electrostatic}
O.~{Penrose}.
\newblock {Electrostatic Instabilities of a Uniform Non-Maxwellian Plasma}.
\newblock {\em Physics of Fluids}, 3:258--265, March 1960.

\bibitem{sakaguchi1988cooperative}
Hidetsugu Sakaguchi.
\newblock Cooperative phenomena in coupled oscillator systems under external
  fields.
\newblock {\em Progr. Theoret. Phys.}, 79(1):39--46, 1988.

\bibitem{sakaguchi1988phase}
Hidetsugu Sakaguchi, Shigeru Shinomoto, and Yoshiki Kuramoto.
\newblock Phase transitions and their bifurcation analysis in a large
  population of active rotators with mean-field coupling.
\newblock {\em Progr. Theoret. Phys.}, 79(3):600--607, 1988.

\bibitem{strogatz2000from}
Steven~H. Strogatz.
\newblock From {K}uramoto to {C}rawford: exploring the onset of synchronization
  in populations of coupled oscillators.
\newblock {\em Phys. D}, 143(1-4):1--20, 2000.
\newblock Bifurcations, patterns and symmetry.

\bibitem{strogatz1991stability}
Steven~H. Strogatz and Renato~E. Mirollo.
\newblock Stability of incoherence in a population of coupled oscillators.
\newblock {\em J. Statist. Phys.}, 63(3-4):613--635, 1991.

\bibitem{strogatz1992coupled}
Steven~H. Strogatz, Renato~E. Mirollo, and Paul~C. Matthews.
\newblock Coupled nonlinear oscillators below the synchronization threshold:
  relaxation by generalized {L}andau damping.
\newblock {\em Phys. Rev. Lett.}, 68(18):2730--2733, 1992.

\bibitem{vanderbauwhede1989centre}
A.~Vanderbauwhede.
\newblock Centre manifolds, normal forms and elementary bifurcations.
\newblock In {\em Dynamics reported, {V}ol.\ 2}, volume~2 of {\em Dynam.
  Report. Ser. Dynam. Systems Appl.}, pages 89--169. Wiley, Chichester, 1989.

\bibitem{vanderbauwhede1992center}
A.~Vanderbauwhede and G.~Iooss.
\newblock Center manifold theory in infinite dimensions.
\newblock In {\em Dynamics reported: expositions in dynamical systems},
  volume~1 of {\em Dynam. Report. Expositions Dynam. Systems (N.S.)}, pages
  125--163. Springer, Berlin, 1992.

\end{thebibliography}
\end{document}